\documentclass{amsart}
\usepackage{verbatim}
\usepackage{amssymb, amsmath}
\usepackage{amsthm}
\usepackage{graphicx}
\usepackage{caption}
\usepackage{subcaption}
\usepackage[all]{xy}
\usepackage{graphicx}
\usepackage{mathtools}
\usepackage{fancyhdr}
\usepackage{color}
\usepackage{algorithm,algpseudocode}
\usepackage{hyperref}

\newfloat{algorithm}{tbp}{loa}
\providecommand{\algorithmname}{Algorithm}
\floatname{algorithm}{\protect\algorithmname}

\begin{document}
\newtheorem{thm}{Theorem}[section]
\newtheorem*{thm*}{Theorem}
\newtheorem{lem}[thm]{Lemma}
\newtheorem{prop}[thm]{Proposition}
\newtheorem{cor}[thm]{Corollary}
\newtheorem*{conj}{Conjecture}
\newtheorem{proj}[thm]{Project}
\newtheorem{question}[thm]{Question}
\newtheorem{rem}{Remark}[section]
\newtheorem*{lemma*}{Lemma}
\newtheorem*{corollary*}{Corollary}

\theoremstyle{definition}
\newtheorem{defn}{Definition}
\newtheorem*{remark}{Remark}
\newtheorem{exercise}{Exercise}
\newtheorem*{exercise*}{Exercise}
\numberwithin{equation}{section}

\title[Asymptotic Independence For Surface Groups]{Asymptotic Independence for Random Permutations from Surface Groups}
\author{Yotam Maoz}
\address{School of Mathematical Sciences, Tel Aviv University, Tel Aviv 69978, Israel}
\email{yotammaoz@mail.tau.ac.il}
\date{\today}

\begin{abstract}
    Let $X$ be an orientable hyperbolic surface of genus $g\geq 2$ with a marked point $o$,
    and let $\Gamma$ be an orientable hyperbolic surface group isomorphic to $\pi_{1}(X,o)$.
    Consider the space $\text{Hom}(\Gamma,S_{n})$ which corresponds to
    $n$-sheeted covers of $X$ with labeled fiber.
    Given $\gamma\in\Gamma$ and a uniformly random $\phi\in\text{Hom}(\Gamma,S_{n})$,
    what is the expected number of fixed points of $\phi(\gamma)$?

    Formally, let $F_{n}(\gamma)$ denote the number of fixed points of $\phi(\gamma)$ for a uniformly random $\phi\in\text{Hom}(\Gamma,S_{n})$.
    We think of $F_{n}(\gamma)$ as a random variable on the space $\text{Hom}(\Gamma,S_{n})$.
    We show that an arbitrary fixed number of products of the variables $F_{n}(\gamma)$ are asymptotically independent as $n\to\infty$ when there are no obvious obstructions.
    We also determine the limiting distribution of such products.
    Additionally, we examine short cycle statistics in random permutations of the form $\phi(\gamma)$ for a uniformly random $\phi\in\text{Hom}(\Gamma,S_{n})$.
    We show a similar asymptotic independence result and determine the limiting distribution.
\end{abstract}


\maketitle

\pagebreak
\tableofcontents
\pagebreak
\section{Introduction}
Let $X$ be a compact hyperbolic surface with marked point $o$ of genus $g\geq 2$. In addition, let its fundamental group be
\[
    \Gamma = \left\langle a_{1},b_{1},...,a_{g},b_{g}|[a_{1},b_{1}]...[a_{g},b_{b}]\right\rangle.
\]
The main focus of this paper is the space $\text{Hom}(\Gamma,S_{n})$ consisting of all
group homomorphisms from $\Gamma$ to $S_{n}$ \textemdash the symmetric group of permutations of $\{1,...,n\}$.
This space admits a topological interpretation. Indeed, each $\phi \in \text{Hom}(\Gamma,S_{n})$ corresponds to a
degree-$n$ covering of $X$ denoted by $X_{\phi}$. In general, for a point $o\in X$ there is a bijection between
$n$-sheeted covers $p:\hat{X}\to X$ of $X$ with labeled fiber $p^{-1}(o)=\{1,...,n\}$ and elements of the space $\text{Hom}(\Gamma,S_{n})$.
This foundational bijection will be discussed further in Section \ref{bk}.

As the space $\text{Hom}(\Gamma,S_{n})$ is finite\footnote{One can show that for a finite group $G$ we have the formula:
\[
    \#\text{Hom}(\Gamma,G) = |G|^{2g-1}\zeta^{G}(2g-2),
\]
where $\zeta^{G}(s) = \sum_{\rho\in \text{Irrep}(G)}\dim(\rho)^{-s}$
is the Witten zeta function of $G$ and $\text{Irrep}(G)$ denotes the set of complex irreducible representations of $G$. See \cite[Proposition 3.2]{LS}.
}, we endow it with the uniform probability measure and view it as a probability space.
Given a random variable $T$ on this space, we define the expected value operator as:
\[
    \mathbb{E}_{n}[T]=\frac{1}{\#\text{Hom}(\Gamma,S_{n})}\sum_{\phi\in\text{Hom}(\Gamma,S_{n})}T(\phi).
\]
Thus, we have a notion of a random cover of the space $X$.
In particular, fixing $X$ and $n$, one obtains a notion of a random surface.

Over time, there have been a few works regarding the probability space $\text{Hom}(\Gamma,S_{n})$.
For example, Liebeck and Shalev \cite[Theorem 1.12]{LS} showed that for a uniformly random $\phi \in \text{Hom}(\Gamma,S_{n})$ a.a.s., i.e. with probability
tending to $1$ as $n\to\infty$, we have $A_{n}\subseteq Im(\phi)$. The image of a random $\phi\in \text{Hom}(\Gamma,S_{n})$
is a.a.s. transitive and so the corresponding cover $X_{\phi}$ is a.a.s. connected.

Let $l\subseteq X$ be a loop based at $o$, that is, a map $l:[0,1]\to X$ such that $l(0)=l(1)=o$. In addition let $p:\hat{X}\to X$ be an $n$-cover of $X$ with labeled fiber $p^{-1}(o) = \{1,...,n\}$.
For each $i\in \{1,...,n\}$ one can lift $l$ to a path $l_{i}$ in $\hat{X}$ based at $i$, that is, for every $i\in \{1,...,n\}$ one can find a map $l_{i}:[0,1]\to \hat{X}$ such that $l_{i}(0)=i$ and $l = p\circ l_{i}$.
Recently,  Magee and Puder \cite{MP2,MP1} considered the following natural question one could ask of the space $\text{Hom}(\Gamma,S_{n})$:

\begin{question}\label{q org}
    Fix a loop $l \subseteq X$ based at $o$. For a random $n$-cover $X_{\phi}$ of $X$, what is the expected number of times the lift $l_{i}$ is itself a loop in the cover?
\end{question}

As the space $\text{Hom}(\Gamma,S_{n})$ corresponds to $n$-covers of $X$ with labeled fiber, the above question has an algebraic equivalent:
\begin{question}\label{q}
    For a fixed $\gamma\in \Gamma$,
    what is the expected number of fixed points of the permutation $\phi(\gamma)$ as $\phi$ varies over the space $\text{Hom}(\Gamma,S_{n})$?
\end{question}

Let us now formalize the above and give a bit of notation.
Given a permutation $\sigma\in S_{n}$ we denote by $\text{Fix}(\sigma)$ its set of fixed points when acting on $\{1,...,n\}$.
Each $\gamma \in \Gamma$ defines a random variable on the space $\text{Hom}(\Gamma,S_{n})$ by $\#\text{Fix}(\phi(\gamma))$.
We denote this random variable by $F_{n}(\gamma)$.
One can generalize the variables $F_{n}(\gamma)$. Each pair $(\gamma,d)$, where $\gamma \in \Gamma$ and $d$ is a positive integer, defines a random variable on the space which we denote by $C_{n,d}(\gamma)$.
As a random variable, $C_{n,d}(\gamma)$ is the number of $d$-cycles in the permutation $\phi(\gamma)$ for a random $\phi \in \text{Hom}(\Gamma,S_{n})$.
In particular $C_{n,1}(\gamma) = F_{n}(\gamma)$.

While trying to tackle Question \ref{q}, one notes some important observations:
\begin{enumerate}
    \item
          If $\gamma = \gamma_{0}^{a}$ for some $\gamma_{0}\in \Gamma$ and some positive integer power $a$, then:
          \[
              F_{n}(\gamma) = \sum_{1\leq d|a}dC_{n,d}(\gamma_{0}).
          \]

    \item
          The values of $F_{n}(\gamma)$ and of $C_{n,d}(\gamma)$ depend only on the conjugacy class of $\gamma$ in $\Gamma$, as conjugate elements in $S_{n}$ have the same cycle structure.

    \item
          We have:
          \[
              F_{n}(\gamma) = F_{n}(\gamma^{-1}),
          \]
          as a permutation and its inverse have the same cycle structure.
\end{enumerate}

Due to these observations, we denote by $\mathcal{P}$
the set of primitive conjugacy classes in $\Gamma$ different from the identity.
By primitive we mean that they are conjugacy classes of non-power elements.
We let $\mathcal{P}_{0} = \mathcal{P}/\sim$ where $[\gamma] \sim [\delta]$ iff $\gamma = \delta$ or $\gamma = \delta^{-1}$.
These sets also admit nice topological/geometric interpretations.
The set $\mathcal{P}$ corresponds to primitive oriented geodesics on $X$ while $\mathcal{P}_{0}$ corresponds to primitive non-oriented geodesics on $X$.
Using some abuse of notation, we write $\gamma \in \mathcal{P}$ or $\gamma \in \mathcal{P}_{0}$ to mean that the conjugacy class of $\gamma$
is an element of $\mathcal{P}$ respectively $\mathcal{P}_{0}$. With this notation in mind, note that saying that two elements $\gamma,\delta\in \mathcal{P}_{0}$ are
distinct means that $\gamma$ is not conjugate to $\delta$ or $\delta^{-1}$. This will come up a lot in our main theorems.

\subsection{Recent Progress}
As for the state of Question \ref{q}, we know quite a lot.
Recently, Puder and Magee \cite[Theorem 1.2]{MP1} gave the following result:
\begin{thm}\label{basic}
    For $1\neq \gamma \in \Gamma$ write $\gamma = \gamma_{0}^{a}$ for primitive $\gamma_{0}\in \Gamma$ and $a\in\mathbb{Z}_{\geq 1}$, then as $n\to \infty$:
    \[
        \mathbb{E}_{n}[F_{n}(\gamma)] =  d(a) + O_{\gamma}(1/n),
    \]
    where $d(a)$ is the number of positive divisors of $a$.
\end{thm}

Even more recently, Puder and Zimhony gave the following stronger result:
\begin{thm}[Corollary 1.7 of \cite{PZ}]\label{tpz}
    For $1\neq \gamma \in \Gamma$ write $\gamma = \gamma_{0}^{a}$ for primitive $\gamma_{0}\in \Gamma$  and $a\in\mathbb{Z}_{\geq 1}$, then as $n\to \infty$:
    \[
        F_{n}(\gamma) \xrightarrow{\text{dis}} \sum_{1\leq d|a}dZ_{1/d},
    \]
    where the $Z_{\lambda}$-s are independent Poisson random variables with parameters $\lambda$, and $``\xrightarrow{\text{dis}}"$ denotes convergance in distribution.
\end{thm}
From now on $Z_{\lambda}$ will always denote a Poisson random variable with parameter $\lambda$.
In addition, whenever we write $\sum_{d|N}$, we mean that the sum is over the positive divisors of the positive integer $N$, i.e. we sum over $d|N$ such that $1\leq d$.

As can be seen, the bound on the error in Theorem \ref{basic} depends on the element $\gamma$.
This sometimes causes problems when trying to bound the error over a range of $\gamma$-s. Magee and Puder, together with
Naud, in \cite{MPN} give the following uniform bound on short $\gamma$-s:
\begin{thm}[Theorem 1.11 in \cite{MPN}]\label{unif}
    There is a constant $A$ (dependent on $g$) such that for any $c>0$, if $1\neq \gamma \in \Gamma$ is primitive and its word length is at most $c\log(n)$ then:
    \[
        \mathbb{E}_{n}[F_{n}(\gamma)] = 1 + O_{c,g}\left(\frac{(\log n)^{A}}{n}\right).
    \]
    The word length of an element $\gamma \in \Gamma$ is the length of a shortest word on the generators $a_{1},b_{1},...,a_{g},b_{g}$ and their inverses representing $\gamma$.
    In particular, the error does not depend on $\gamma$ as long as its word length is $\leq c\log(n)$.
\end{thm}

As for the variables $C_{n,d}(\gamma)$, Puder and Zimhony give a general result that when applied to surface groups, readily gives the following result.
Note that variables being asymptotically independent means that they jointly converge in distribution towards independent random variables.
\begin{thm}[Theorem 1.15 in \cite{PZ}]\label{cycle thm}
    For $1\neq \gamma \in \Gamma$ primitive we have:
    \begin{itemize}
        \item
              For all $d\in\mathbb{Z}_{\geq 1}$ we have $C_{n,d}(\gamma)\xrightarrow{\text{dis}}Z_{1/d}$.
        \item
              $C_{n,1}(\gamma),C_{n,2}(\gamma),...,C_{n,t}(\gamma)$ are asymptotically
              independent for fixed $t$.
    \end{itemize}
\end{thm}

The question we seek to tackle in this paper is that of independence for the variables $F_{n}(\gamma)$. To this end, Puder and Zimhony give a general result that when applied to surface groups, yields the following:
\begin{thm}[Theorem 1.14 in \cite{PZ}]\label{indep}
    Let $1\neq \gamma,\delta \in \Gamma$ and set $\gamma = \gamma_{0}^{a}$ and $\delta = \delta_{0}^{b}$ for primitive $\gamma_{0},\delta_{0}\in \Gamma$ and $a,b\in \mathbb{Z}_{\geq 1}$.
    TFAE:
    \begin{enumerate}
        \item
              The variables $F_{n}(\gamma)$ and $F_{n}(\delta)$ are asymptotically independent as \newline $n\to \infty$.
        \item
              $\gamma_{0}$ and $\delta_{0}$ are distinct as elements in $\mathcal{P}_{0}$, that is, $\gamma_{0}$ is not conjugate to $\delta_{0}$ or to $\delta_{0}^{-1}$.
        \item
              We have:
              \[
                  \mathbb{E}_{n}[F_{n}(\gamma)F_{n}(\delta)] = \mathbb{E}_{n}[F_{n}(\gamma)]\mathbb{E}_{n}[F_{n}(\delta)] + O_{\gamma,\delta}(1/n).
              \]
    \end{enumerate}
\end{thm}

\subsection{Main Results}
In this paper, we seek to generalize the implications \newline $(2)\implies (1)$ and $(2)\implies (3)$ in Theorem \ref{indep}, and the statements in Theorem \ref{cycle thm}.
We show the following main theorem:
\begin{thm}\label{main thm}
    Let $\gamma_{1},...,\gamma_{t}\in\mathcal{P}_{0}$
    be distinct and for each $i$ let $r_{i}\geq 1$ and:
    \[
        a_{i,1},...,a_{i,r_{i}}\geq1,
    \]
    be integers. As $n\to\infty$ we have:

    \[
        \mathbb{E}_{n}\left[\prod_{i=1}^{t}\prod_{j=1}^{r_{i}}F_{n}(\gamma_{i}^{a_{i,j}})\right]=\prod_{i=1}^{t}\mathbb{E}_{n}\left[\prod_{j=1}^{r_{i}}F_{n}(\gamma_{i}^{a_{i,j}})\right] + O(1/n).
    \]
    with the implied constant depending on  $\gamma_{1},...,\gamma_{t}$ and the integers $a_{i,j}$.
\end{thm}

As a corollary, we get:
\begin{cor} \label{main cor}
    In the same setting as Theorem \ref{main thm} we have:
    \[
        \mathbb{E}_{n}\left[\prod_{i=1}^{t}\prod_{j=1}^{r_{i}}F_{n}(\gamma_{i}^{a_{i,j}})\right]=\prod_{i=1}^{t}\underset{n\to\infty}{\lim}\mathbb{E}_{n}\left[\prod_{j=1}^{r_{i}}F_{n}(\gamma_{i}^{a_{i,j}})\right]+ O(1/n),
    \]
    and:
    \[
        \underset{n\to\infty}{\lim}\mathbb{E}_{n}\left[\prod_{i=1}^{t}\prod_{j=1}^{r_{i}}F_{n}(\gamma_{i}^{a_{i,j}})\right]=\prod_{i=1}^{t}\underset{n\to\infty}{\lim}\mathbb{E}_{n}\left[\prod_{j=1}^{r_{i}}F_{n}(\gamma_{i}^{a_{i,j}})\right].
    \]
\end{cor}

As an additional corollary, we get:
\begin{cor}\label{moments cor}
    In the same setting as Theorem \ref{main thm} for each positive integer $k$ and $1\leq i \leq t$ let $Z_{1/k}^{(i)}$ be
    a Poisson random variable with parameter $1/k$, such that all $Z$-s are independent (in the strong sense, not just pairwise independent). Define:
    \[
        X^{(i)}_{a_{i,1},...,a_{i,r_{i}}} = \prod_{j=1}^{r_{i}}\sum_{k|a_{i,j}}kZ_{1/k}^{(i)},
    \]
    and note that for different $i$ the variables $X^{(i)}_{a_{i,1},...,a_{i,r_{i}}}$ are independent.
    Then the cross moments of:
    \[
        \prod_{j=1}^{r_{1}}F_{n}(\gamma_{1}^{a_{1,j}}),...,\prod_{j=1}^{r_{t}}F_{n}(\gamma_{t}^{a_{t,j}}),
    \]
    and of:
    \[
        X^{(1)}_{a_{1,1},...,a_{1,r_{1}}},...,X^{(t)}_{a_{t,1},...,a_{t,r_{t}}},
    \]
    are asymptotically equal. That is, for every $s_{1},...,s_{t}\in \mathbb{Z}_{\geq 1}$ we have:
    \begin{multline}\label{cross eq}
        \underset{n\to\infty}{\lim}\mathbb{E}_{n}\left[\left(\prod_{j=1}^{r_{1}}F_{n}(\gamma_{1}^{a_{1,j}})\right)^{s_{1}}\cdot...\cdot\left(\prod_{j=1}^{r_{t}}F_{n}(\gamma_{t}^{a_{t,j}})\right)^{s_{t}}\right] = \\
        \mathbb{E}\left[\left(X^{(1)}_{a_{1,1},...,a_{1,r_{1}}}\right)^{s_{1}}\cdot...\cdot \left(X^{(t)}_{a_{t,1},...,a_{t,r_{t}}}\right)^{s_{t}}\right] = \\
        \mathbb{E}\left[\left(X^{(1)}_{a_{1,1},...,a_{1,r_{1}}}\right)^{s_{1}}\right] \cdot...\cdot\mathbb{E}\left[\left(X^{(t)}_{a_{t,1},...,a_{t,r_{t}}}\right)^{s_{t}}\right].
    \end{multline}
\end{cor}

Using Corollary \ref{moments cor} we show the following generalization of Theorem \ref{cycle thm}:
\begin{thm}\label{cyc ext}
    Let $\gamma_{1},...,\gamma_{t}\in \mathcal{P}_{0}$ be distinct. For each positive integer $k$ and $1\leq i \leq t$ let $Z_{1/k}^{(i)}$ be
    a Poisson random variable with parameter $1/k$, such that all $Z$-s are independent (in the strong sense, not just pairwise independent).
    In addition for each $i$ let $r_{i}\geq 1$ be an integer.
    Then:
    \begin{multline}\label{cyc dist eq}
        \left(C_{n,1}(\gamma_{1}),...,C_{n,r_{1}}(\gamma_{1}),...,C_{n,1}(\gamma_{t}),...,C_{n,r_{t}}(\gamma_{t})\right)\\
        \stackrel{\text{dis}}{\rightarrow}
        \left(Z^{(1)}_{1},...,Z^{(1)}_{1/r_{1}},...,Z^{(t)}_{1},...,Z^{(t)}_{1/r_{t}}\right).
    \end{multline}
    That is, the number of bounded cycles for different elements in $\mathcal{P}_{0}$ converge jointly to their respective Poisson variable. In particular, the variables in the LHS of (\ref{cyc dist eq}) are asymptotically independent.
\end{thm}

Finally, we show the following stronger version of Theorem \ref{tpz}:
\begin{thm}\label{dist thm}
    In the same setting as Corollary \ref{moments cor}, as $n\to\infty$ we have:
    {\footnotesize{\begin{multline}\label{dist eq}
            \left(\prod_{j=1}^{r_{1}}F_{n}(\gamma_{1}^{a_{1,j}}),...,\prod_{j=1}^{r_{t}}F_{n}(\gamma_{t}^{a_{t,j}})\right)
            \stackrel{\text{dis}}{\rightarrow}
            \left(X^{(1)}_{a_{1,1},...,a_{1,r_{1}}},...,X^{(t)}_{a_{t,1},...,a_{t,r_{t}}}\right) = \\
            \left(\prod_{j=1}^{r_{1}}\sum_{k|a_{1,j}}kZ_{1/k}^{(1)},...,\prod_{j=1}^{r_{t}}\sum_{k|a_{t,j}}kZ_{1/k}^{(t)}\right).
        \end{multline}}}
    In particular, the variables
    \[
        \prod_{j=1}^{r_{1}}F_{n}(\gamma_{1}^{a_{1,j}}),...,\prod_{j=1}^{r_{t}}F_{n}(\gamma_{t}^{a_{t,j}}),
    \]
    are asymptotically independent as $n\to \infty$.
\end{thm}
Note that each variable $Z_{1/k}^{(i)}$ may appear several times in (\ref{dist eq}).

Consider the following example.
Let $\gamma,\delta\in \mathcal{P}_{0}$ be distinct and suppose we wish to estimate $\mathbb{E}_{n}\left[F_{n}(\gamma^{2})F_{n}(\gamma^{3})F_{n}(\delta^{4})\right]$ for large $n$.
Using Corollary \ref{main cor} we have:
\begin{multline*}
    \mathbb{E}_{n}\left[F_{n}(\gamma^{2})F_{n}(\gamma^{3})F_{n}(\delta^{4})\right] = \\
    \underset{n\to\infty}{\lim}\mathbb{E}_{n}\left[F_{n}(\gamma^{2})F_{n}(\gamma^{3})\right]\underset{n\to\infty}{\lim}\mathbb{E}_{n}\left[F_{n}(\delta^{4})\right] + O_{\gamma,\delta}(1/n).
\end{multline*}
Using Theorem \ref{dist thm} for $F_{n}(\gamma^{2})F_{n}(\gamma^{3})$ and $F_{n}(\delta^{4})$ we have:
\begin{multline*}
    F_{n}(\gamma^{2})F_{n}(\gamma^{3})\stackrel{\text{dis}}{\to} \left(Z^{(1)}_{1}+2Z^{(1)}_{1/2}\right)\left(Z^{(1)}_{1} + 3Z^{(1)}_{1/3}\right) =\\
    \left(Z_{1}^{(1)}\right)^{2} + 2Z^{(1)}_{1}Z^{(1)}_{1/2} + 3Z^{(1)}_{1}Z^{(1)}_{1/3} + 6Z^{(1)}_{1/2}Z^{(1)}_{1/3},
\end{multline*}
and:
\[
    F_{n}(\delta^{4})\stackrel{\text{dis}}{\to} Z^{(2)}_{1}+2Z^{(2)}_{1/2}+4Z^{(2)}_{1/4}.
\]
Recalling that $\mathbb{E}\left[Z_{\lambda}^{2}\right] = \lambda^{2} + \lambda$ we get:
\begin{multline*}
    \underset{n\to\infty}{\lim}\mathbb{E}_{n}\left[F_{n}(\gamma^{2})F_{n}(\gamma^{3})\right]\underset{n\to\infty}{\lim}\mathbb{E}_{n}\left[F_{n}(\delta^{4})\right] =\\
    \left(1+1 + 2\cdot1\cdot\frac{1}{2} + 3\cdot1\cdot\frac{1}{3} + 6\cdot\frac{1}{2}\cdot\frac{1}{3}\right)\cdot\left(1 + 2\cdot\frac{1}{2} + 4\cdot\frac{1}{4}\right) =15.
\end{multline*}
Therefore:
\[
    \mathbb{E}_{n}\left[F_{n}(\gamma^{2})F_{n}(\gamma^{3})F_{n}(\delta^{4})\right] = 15 + O_{\gamma,\delta}(1/n).
\]

\subsection{General Motivation and Laplace Eigenvalues}
The study of the variables $F_{n}(\gamma)$ and $C_{n,d}(\gamma)$ is partially motivated by the study of Laplace eigenvalues on hyperbolic surfaces.
To give a bit of context, the hyperbolic Laplace operator $\Delta$ on $X$ has an extension to a non-negative, unbounded, self-adjoint operator on $L^{2}(X)$ whose spectrum consists of real eigenvalues:
\[
    0=\lambda_{0}(X)\leq \lambda_{1}(X) \leq ...\leq \lambda_{n}(X) \leq ...
\]
with $\lambda_{n}(X)\to\infty$ as $n\to \infty$.

There are long-standing conjectures concerning this spectrum. A famous one is Selberg's eigenvalue conjecture \cite{SELB} which, broadly speaking, says that certain (arithmetic) surfaces (orbifolds) have no eigenvalues below $\frac{1}{4}$.
Theorem \ref{unif} originally appears in \cite{MPN} where they show that a.a.s., a random $n$-cover of $X$ has no new eigenvalues below $\frac{3}{16}-\varepsilon$.

The motivation for the current paper comes from considering a different aspect of the spectrum, namely, the distribution of the high-energy eigenvalues.
A conjecture of Berry \cite{BR2,BR1} says that certain statistics of the spectrum display the statistics found in random matrix theory, specifically GOE/GUE statistics.
Until recently, there has not been much progress toward this conjecture. Lately, however, there have been somewhat successful attempts of proving similar claims for random surfaces.

Recently, Rudnick and Wigman \cite{Ru,RW2,RW} considered similar questions to Berry's conjecture for random surfaces. They show results that support the validity of Berry's conjecture. Their model for a random surface is a natural one that places a probability measure (the Weil-Peterson measure) on the moduli space of compact genus $g$ surfaces without boundary.

Similarly, Naud \cite{Naud} also considered a question similar to Berry's conjecture and gave a result in the spirit of Rudnick \cite{Ru}. This time, his model of a random surface is the one present in the current paper. First, fix a compact hyperbolic surface $X$ as we do, then consider $n$-covers of $X$ for large $n$. This construction gives an alternative, also natural, model for a random surface. In this random cover model, when trying to attack questions similar to Berry's conjecture, a lot of the randomness boils down to understanding the random variables $F_{n}(\gamma)$ for various $\gamma$ and large $n$.

Using the results found here, we show results similar to Rudnick and Wigman's for the random cover model. See \cite{eigenval}.
Some are essentially the same result but in a different model, while others are similar to the results in \cite{RW2} and are slightly stronger.

\subsection{Overview of the Paper}
We first give formal background in Section \ref{bk} and try to introduce the reader to the language and tools developed by Magee, Puder, and Zimhony \cite{MP2,MP1,PZ}
used to analyze the variables $F_{n}(\gamma)$ and $C_{n,d}(\gamma)$. We show how one can generalize the questions asked of $F_{n}(\gamma)$ to broader questions
asked of objects dubbed \textit{subcovers} of $X$. We then give some results of Magee, Puder, and Zimhony related to the study of these subcovers.

After the relevant background, we prove the main Theorem \ref{main thm} and its corollaries assuming an important Lemma \ref{imp lemma} which is at the core of our results. Next, we give more background ahead of the proof of Lemma \ref{imp lemma}, after which we give the proof.

\section*{Acknowledgements}
I would like to thank Prof. Doron Puder for some helpful conversations and insights surrounding this work and for reviewing earlier drafts.
I would also like to thank my advisor Prof. Ze\'{e}v Rudnick for reviewing earlier drafts and for his general guidance.
\thanks{This work was supported by the Israel Science Foundation (grant No. 1881/20) and the ISF-NSFC joint research program (grant No. 3109/23).}

\section{Background}\label{bk}
\subsection{Covers of \texorpdfstring{$X$}{Lg}}\label{cvrs}
As stated earlier, for an arbitrary $o\in X$ the $n$-sheeted covers of $X$ with a labeled fiber of $o$ are in bijection with
$\text{Hom}(\Gamma,S_{n})$. See \newline \cite[pp. 68-70]{HT} for a comprehensive discussion. To see this bijection, let $\mathbb{H}$
be the standard hyperbolic plane with constant curvature -1. In addition
view $\Gamma$ as $\pi_{1}(X,o)$ and as a subgroup of $PSL_{2}(\mathbb{R})$ - the orientation preserving automorphism
group of $\mathbb{H}$. Also view $X$ as $\Gamma\backslash\mathbb{H}$ and set $[n] = \{1,...,n\}$.

Given $\phi\in\text{Hom}(\Gamma,S_{n})$ we define an action of
the discrete $\Gamma$ on $\mathbb{H}\times[n]$ by:
\[
    \gamma(z,j)=(\gamma z,\phi(\gamma)j).
\]
Quotienting $\mathbb{H}\times[n]$ by this action, one finds a (possibly
not connected) $n$-sheeted cover $X_{\phi}=\Gamma\backslash\left(\mathbb{H}\times[n]\right)$
of $X$ (with the covering map being the projection on $X=\Gamma\backslash\mathbb{H})$.
The bijection
\[
    \{n\text{-sheeted covers of }X\}\longleftrightarrow\text{Hom}(\Gamma,S_{n}),
\]
is then given by mapping $\phi$ to $X_{\phi}$.

To see the inverse of this bijection, let $p:\hat{X}\to X$ be an $n$-sheeted
cover of $X$ and label $p^{-1}(o)=\{1,...,n\}$. Let $\gamma\in\pi_{1}(X,o)$, which we consider as a map $\gamma:[0,1]\to X$ for which $\gamma(0)=\gamma(1)=o$,
and let $i\in\{1,...,n\}=p^{-1}(o)$. One can uniquely lift $\gamma$ to a path in
$\hat{X}$ starting at $i$, that is, one can find a map $\hat{\gamma}:[0,1]\to \hat{X}$ for which $\hat{\gamma}(0)=i$ and $\gamma = p\circ\hat{\gamma}$.
As $\gamma(1)=o$ we must have that $\hat{\gamma}(1)\in p^{-1}(o)=\{1,...,n\}$, that is, the endpoint of the lift $\hat{\gamma}$ lies in $\{1,...,n\}$.
We denote this point as $\phi(\gamma)i$.
One also notes that the map $\phi(\gamma)$ from
$\{1,...,n\}$ to itself is actually a permutation, as its inverse
is $\phi(\gamma^{-1})$ where $\gamma^{-1}$ as a loop in $X$ is
just $\gamma$ with reversed orientation. The map $\gamma\mapsto\phi(\gamma)$
is trivially a group homomorphism $\Gamma\to S_{n}$ when $S_{n}$ is equipped with the group structure of composing permutations from left to right.
This is the map:
\[
    \{n\text{-sheeted covers of }X\}\to\text{Hom}(\Gamma,S_{n}),
\]
in the bijection:
\[
    \{n\text{-sheeted covers of }X\}\longleftrightarrow\text{Hom}(\Gamma,S_{n}).
\]

\subsection{Subcovers of \texorpdfstring{$X$}{Lg} and Related Tools}
We now start introducing the language of \textit{subcovers} of $X$ and the related tools we need for our proofs.
See \cite{MP2,MP1,PZ} for a more comprehensive background.

For a given compact hyperbolic surface $X$, one can endow $X$ with a CW-structure
obtained from gluing the edges of a $4g$-gon according to the pattern:
\[
    a_{1},b_{1},a_{1}^{-1},b_{1}^{-1},a_{2},...,a_{g}^{-1},b_{g}^{-1},
\]
so that the CW-structure on $X$ consists of a single vertex, $2g$
directed edges (1-cells) and one 2-cell, see Figure \ref{fig 1}. From now on we denote by $o$ the single vertex in the CW-structure of $X$
and identify $\Gamma$ with $\pi_{1}(X,o)$.
Also, notice that every cover
of $X$ inherits the CW-structure from $X$ (as open cells are covered
by disjoint open sets).

Next up is the definition of a \textit{subcover} (sub-cover) of $X$.
This is the fundamental object we work with throughout the entire rest of the paper.
\begin{defn} [Subcovers of $X$]
    A subcover $Y$ of $X$ is a (not
    necessarily connected) sub-complex of a (not necessarily finite degree)
    covering space of $X$. For any subcover $Y$ we attach a map to $X$, namely the restricted
    covering map $p:Y\to X$ which is an immersion. A morphism of subcovers
    $f:Y_{1}\to Y_{2}$ is a combinatorial morphism of CW-complexes that
    commutes with the restricted covering maps $p_{1}:Y_{1}\to X$ and $p_{2}:Y_{2}\to X$,
    so that $p_{2}\circ f=p_{1}$.
\end{defn}
We think of a subcover $Y$ of $X$ and its attached restricted covering map $p:Y\to X$ as essentially the same.

Intuitively, it helps to think of subcovers of $X$ as (at most) 2-dimensional CW-complexes with edge labels and directions.
The morphisms between such subcovers are the natural ones respecting edge labels and directions. See Figures \ref{fig 1} and \ref{fig 2} for examples.
We can formalize this intuition. To this end, we give an alternative definition of subcovers that comes
from \cite{MP2}. Note that in \cite{MP2} they refer to subcovers as "tiled surfaces''.

\begin{defn} [Alternative definition for subcovers of $X$] \label{alt def}
    A subcover $Y$ of $X$ is an at-most two-dimensional CW-complex
    with an assignment of both a direction and a label from $\{a_{1},b_{1},...,a_{g},b_{g}\}$
    to each edge, such that:

    \begin{itemize}
        \item
              Every vertex of $Y$ has at most one incoming $l$-labeled edge and
              at most one outgoing $l$-labeled edge for each $l\in\{a_{1},b_{1},...,a_{g},b_{g}\}$.
        \item
              Every path in the 1-skeleton of $Y$ reading a word in the letters:
              \[
                  \{a_{1}^{\pm1},b_{1}^{\pm1},...,a_{g}^{\pm1},b_{g}^{\pm1}\}
              \]
              which equals the identity in $\Gamma$ must be closed.
        \item
              Every 2-cell in $Y$ is a $4g$-gon glued along a closed path reading
              the relation $[a_{1},b_{1}]...[a_{g},b_{b}]$, and every such closed
              path is the boundary of at most one $4g$-gon.
    \end{itemize}
    In this setting, a morphism between subcovers $A,B$ is a combinatorial map $A \to B$, meaning that it maps $i$-cells to $i$-cells and $l$-labeled edges to $l$-labeled edges while respecting the direction of the edges.
\end{defn}

A concrete example of subcovers which stars in the rest of the paper are subcovers of the form $Y_{\gamma}$ as defined below:
\begin{defn} \label{def Y gamma}[$p_{\gamma}:Y_{\gamma}\to X$]
    For a given $1\neq\gamma\in\Gamma$ fix a shortest word $w_{\gamma}$ on the generators of $\Gamma$ in the conjugacy class $[\gamma]$.
    Let $Y_{\gamma}$ be a cycle subdivided into $k$ edges where $k$ is the length of $w_{\gamma}$.
    The subcover $p_{\gamma}:Y_{\gamma}\to X$ is given by sending
    all vertices to the single vertex in $X$ and all edges to their respective
    edges in $X$ according to the order of the letters in the word $w_{\gamma}$.
\end{defn}
See Figure \ref{fig 1} for an example.

\begin{figure}
    \includegraphics[width=\linewidth]{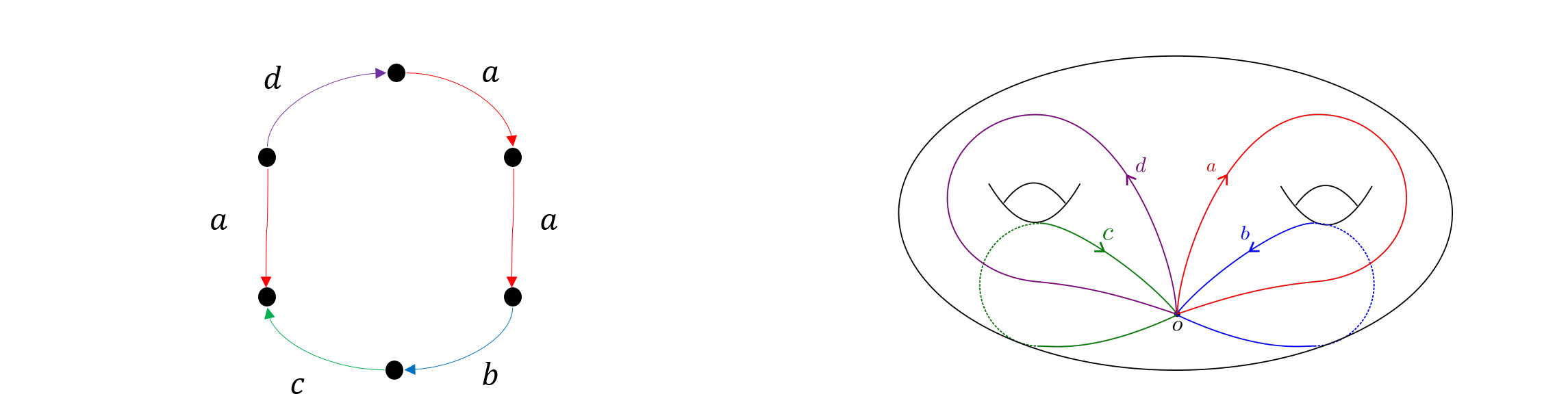}
    \caption{On the right we have a genus-2 surface with its CW-structure (from \cite{MP1}). On the left, we have a subcover $Y$ of a genus-2 surface.
    The subcover $Y$ is the subcover $Y_{\gamma}$ (see Definition \ref{def Y gamma}) of the word $\gamma = a^{2}bca^{-1}d$.
    }
    \label{fig 1}
\end{figure}

Before proceeding, let us show the connection between subcovers and the variables $F_{n}(\gamma)$.
Recall that $o$ is the single vertex in the CW structure of $X$ and consider an element $\phi\in \text{Hom}(\Gamma,S_{n})$.
Let $X_{\phi}$ be the cover corresponding to $\phi$ under
the bijection described in Subsection \ref{cvrs}, and let $p:X_{\phi}\to X$ be the covering map. Recall that we label $p^{-1}(o)$ by $\{1,...,n\}$.
For an element $1\neq \gamma \in \Gamma$ consider the subcover $p_{\gamma}:Y_{\gamma}\to X$, and let $y$ be the vertex in $Y_{\gamma}$ that sits
between the first and last letters of $w_{\gamma}$ on the cycle.

Suppose that $p_{\gamma}$ lifts to a map $\hat{p}_{\gamma}:Y_{\gamma}\to X_{\phi}$. As $\hat{p}_{\gamma}$ is a lift of $p_{\gamma}$, it maps $y\in Y_{\gamma}$ to a vertex $i\in\{1,...,n\}$. Let $l\subseteq X$ be the cycle in the 1-skeleton of $X$ spelling out $w_{\gamma}$, and consider the lift of $l$ to a path in the 1-skeleton of $X_{\phi}$ that starts at $i$. This lift is a path in the 1-skeleton of $X_{\phi}$ that starts at $i$, ends at $\phi(\gamma)i\in\{1,...,n\}$ (see Subsection \ref{cvrs}) and spells out the word $w_{\gamma}$.
On the other hand, as $\hat{p}_{\gamma}$ is a lift of $p_{\gamma}$, the image of $Y_{\gamma}$ in $X_{\phi}$ is a cycle based at $i$ that spells out the word $w_{\gamma}$. As the lift of $l$ to a path in $X_{\phi}$ that starts at $i$ is unique, we conclude that $\phi(\gamma)i = i$, i.e. that $i$ is a fixed point of $\phi(\gamma)$.

Conversely, for each fixed point $i$ of $\phi(\gamma)$, by the description of the cover $X_{\phi}$ in Subsection \ref{cvrs}, the lift of $l$ to a path in $X_{\phi}$ based at $i$ is a cycle that spells out the word $w_{\gamma}$. As such, one can lift $p_{\gamma}$ to a morphism $\hat{p}_{\gamma}:Y_{\gamma}\to X_{\phi}$ that maps $y\in Y_{\gamma}$ to $i$. This shows that lifts of the morphism $p_{\gamma}:Y_{\gamma}\to X$ to a morphism $\hat{p}_{\gamma}:Y_{\gamma}\to X_{\phi}$ correspond bijectively to fixed points of the permutation $\phi(\gamma)$.

This leads us to the following definition which generalizes the problem above. Instead of considering lifts $\hat{p}_{\gamma}:Y_{\gamma}\to X_{\phi}$ of a subcover of the form  $p_{\gamma}:Y_{\gamma}\to X$, we study lifts $\hat{p}:Y\to X_{\phi}$ of a subcover $p:Y\to X$ for a general subcover $Y$.

\begin{defn} [$\mathbb{E}_{Y}(n)$ and $\mathbb{E}_{Y}^{\text{emb}}(n)$]
    For a compact subcover $p:Y\to X$
    we define $\mathbb{E}_{Y}(n)$ to be the expected number of lifts $\hat{p}:Y\to X_{\phi}$
    for a random $\phi\in\text{Hom}(\Gamma,S_{n})$.

    Similarly,
    we define $\mathbb{E}_{Y}^{\text{emb}}(n)$ to be the expected number of injective lifts $\hat{p}:Y\hookrightarrow X_{\phi}$
    for a random $\phi\in\text{Hom}(\Gamma,S_{n})$.
\end{defn}
The discussion before the definition above shows that for a fixed $\gamma\in\Gamma$ and every $\phi\in\text{Hom}(\Gamma,S_{n})$, the number of
lifts of $p_{\gamma}:Y_{\gamma}\to X$ to $\hat{p}_{\gamma}:Y_{\gamma}\to X_{\phi}$ is equal to $\#\text{Fix}(\phi(\gamma))$.
Thus, using this new definition, for $1\neq\gamma\in\Gamma$ we have:
\[
    \mathbb{E}_{n}[F_{n}(\gamma)] = \mathbb{E}_{Y_{\gamma}}(n).
\]

The next definition is of \textit{resolutions} of a given subcover.
These objects essentially track all the ways in which the given subcover can map into full covers of $X$.
\begin{defn} [Resolutions] Let $p:Y\to X$ be a subcover.
    A resolution $R$ of $Y$ is a collection of morphisms of subcovers
    $\{f:Y\to Z_{f}\}$ such that every subcover morphism $h:Y\to\hat{X}$,
    where $\hat{X}$ is a full cover of $X$, decomposes uniquely as:
    \[
        Y \xrightarrow{f} Z_{f} \stackrel{\bar{h}}{\hookrightarrow} \hat{X},
    \]
    with $f\in R$ and $\bar{h}$ an embedding.
    An embedding-resolution
    $R^{\text{emb}}$ of $p:Y\to X$ is a collection of injective morphisms
    of subcovers:
    \[
        \{f:Y\hookrightarrow Z_{f}\},
    \]
    such that for every
    injective subcover morphism $h:Y\hookrightarrow\hat{X}$, where $\hat{X}$
    is a full cover of $X$, there is a unique decomposition $Y\stackrel{f}{\hookrightarrow}Z_{f}\stackrel{\bar{h}}{\hookrightarrow}\hat{X}$
    where $f\in R^{\text{emb}}$ and $\bar{h}$ an embedding.
\end{defn}

\begin{defn} \label{nat res} [Natural resolution] For a compact subcover $p:Y\to X$
    define its natural resolution to be the set of all surjective morphisms
    $\{f:Y\twoheadrightarrow Z_{f}\}$. Note that this resolution is finite as $Y$ is compact.
\end{defn}
See Figure \ref{fig 2} for an example.

A useful lemma which connects the expected values $\mathbb{E}_{Y}$ and $\mathbb{E}_{Y}^{\text{emb}}$
and resolutions is the following:
\begin{lem}[Lemma 2.4 in \cite{PZ}] \label{res and exp}
    For a compact subcover $p:Y\to X$ and $R$ a finite
    resolution of $Y$ we have:
    \[
        \mathbb{E}_{Y}(n)=\sum_{f\in R}\mathbb{E}_{Z_{f}}^{\text{emb}}(n).
    \]
    In addition if $R^{\text{emb}}$ is a finite embedding-resolution of $Y$ then
    we have:
    \[
        \mathbb{E}_{Y}^{\text{emb}}(n)=\sum_{f\in R^{\text{emb}}}\mathbb{E}_{Z_{f}}^{\text{emb}}(n).
    \]
\end{lem}

Next, we need to define the Euler characteristic of a group $G$. Surprisingly, this invariant shows up when studying the asymptotics of the variables $F_{n}(\gamma)$.
We do not give the exact definition (see ~\cite{Bro}), as the only groups we need to consider are the subgroups of $\Gamma$, which by the following proposition (see \cite{Sco}), are of only two kinds:

\begin{prop} \label{subgroups} Let  $H \leq \Gamma = \left\langle a_{1},b_{1},...,a_{g},b_{g}|[a_{1},b_{1}]...[a_{g},b_{b}]\right\rangle$ then $H$ is free or $H\cong\left\langle a_{1},b_{1},...,a_{h},b_{h}|[a_{1},b_{1}]...[a_{h},b_{h}]\right\rangle$ for $h\geq g$, i.e. $H$ is either free or a surface group.
\end{prop}

From here it is enough to define:
\begin{defn} [Euler characteristic of groups]
    Let $G\leq \Gamma$, define the Euler Characteristic of $G$, denoted $\chi(G)$, to be:
    \begin{itemize}
        \item
              $\chi(G)=1-r$  if $G\cong F_{r}$ where $F_{r}$ is the free group on $r$ generators,
        \item
              $\chi(G)=2-2h$ if $G\cong \left\langle a_{1},b_{1},...,a_{h},b_{h}|[a_{1},b_{1}]...[a_{h},b_{h}]\right\rangle$.
    \end{itemize}
\end{defn}

Of central importance in the proof of Theorem ~\ref{main thm}
are those subgroups of $\Gamma$ which have Euler characteristic zero. From
the classification of subgroups of $\Gamma$, Proposition ~\ref{subgroups}, we see that such an Euler characteristic
zero subgroup must be isomorphic to $\mathbb{Z}$.

Next, we define the labeled fundamental group - $\pi_{1}^{\text{lab}}(Y)$, and the group Euler characteristic - $\chi^{\text{grp}}(Y)$, associated with a subcover $p:Y\to X$.
\begin{defn} [$\pi_{1}^{\text{lab}}(Y)$ and $\chi^{\text{grp}}(Y)$]
    For a connected compact subcover \newline $p:Y\to X$ we know that $\pi_{1}(Y)$ is defined up to conjugacy by the choice of a base-point in $Y$,
    thus the image $p_{*}(\pi_{1}(Y))$ in $\pi_{1}(X)=\Gamma$ lies within a well-defined conjugacy class of subgroups of $\Gamma$.
    We denote this conjugacy class of subgroups by $\pi_{1}^{\text{lab}}(Y)$.
    Note that $\pi_{1}^{\text{lab}}(Y)$ is a conjugacy class of f.g. subgroups of $\Gamma$ so we may define $\chi^{\text{grp}}(Y)$
    as the Euler characteristic of some group in $\pi_{1}^{\text{lab}}(Y)$.
    For a compact but not necessarily connected subcover $p:Y\to X$, let
    $Y_{1},...,Y_{n}$ be its connected components. Denote by $\pi_{1}^{\text{lab}}(Y)$
    the multiset $\{\pi_{1}^{\text{lab}}(Y_{1}),...,\pi_{1}^{\text{lab}}(Y_{n})\}$
    and define $\chi^{\text{grp}}(Y)=\sum_{i=1}^{n}\chi(\pi_{1}^{\text{lab}}(Y_{i}))$.
\end{defn}
See Figure \ref{fig 2} for an example.

Another important example comes from the subcovers $p_{\gamma}:Y_{\gamma}\to X$ for \newline
$1\neq \gamma \in \Gamma$. As $Y_{\gamma}$ is topologically just $S^{1}$ we know that $\pi_{1}(Y_{\gamma}) \cong \mathbb{Z}$.
By definition of the map $p_{\gamma}$, the induced map $(p_{\gamma})_{*}:\pi_{1}(Y_{\gamma})\to \Gamma$ must map the positive generator of $\pi_{1}(Y_{\gamma})$ to the element in $\Gamma$ represented by $w_{\gamma}$, which using abuse of notation we also label $w_{\gamma}$, so that:
\[
    \pi_{1}^{\text{lab}}(Y_{\gamma}) = \left[\langle w_{\gamma}\rangle\right] = \left[\langle\gamma\rangle\right],
\]
that is, $\pi_{1}^{\text{lab}}(Y_{\gamma})$ is the conjugacy class of the infinite cyclic subgroup generated by $\gamma$. As such:
\[
    \chi^{\text{grp}}(Y_{\gamma}) = 0.
\]

\begin{figure}
    \includegraphics[width=\linewidth]{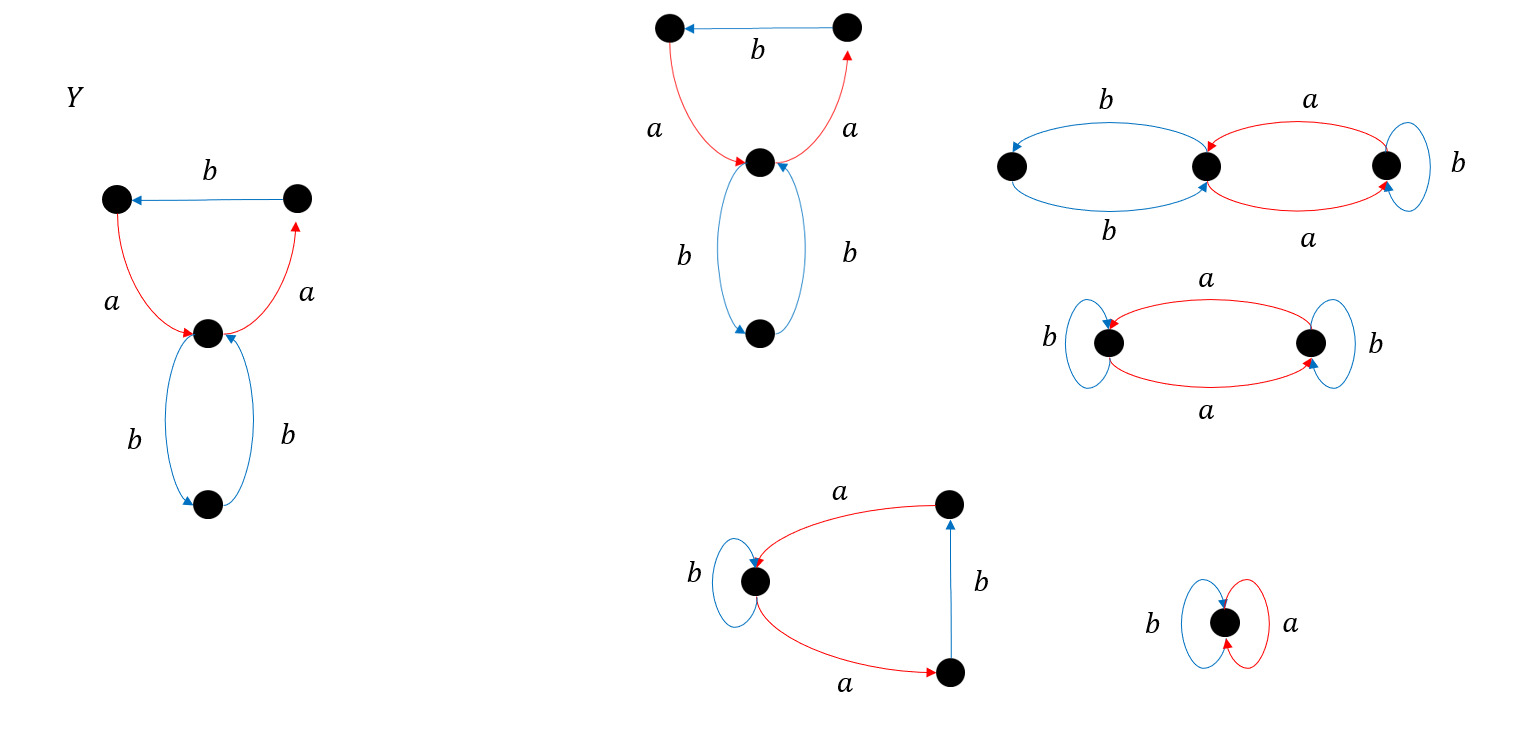}
    \caption{On the left we have a subcover $Y$ of the genus-2 surface with CW-structure as in Figure \ref{fig 1}. On the right, we have its natural resolution. Both subcovers in the top right have $\pi_{1}^{\text{lab}}\cong F_{3}$ and so they have $\chi^{\text{grp}}=-2$.
    The copy of $Y$ and both subcovers on the bottom have $\pi_{1}^{\text{lab}}\cong F_{2}$ which implies $\chi^{\text{grp}}=-1$ for them.
    }
    \label{fig 2}
\end{figure}

Note that if $f:Y\to Z$ is a morphism of compact connected subcovers of $X$
then $\pi_{1}^{\text{lab}}(Y)\leq\pi_{1}^{\text{lab}}(Z)$. This is due to the fact that
$f$ commutes with the projections $Y\xrightarrow{p}X,Z\xrightarrow{q}X$,
and thus so does $f_{*}$. Note that $\pi_{1}^{\text{lab}}(Y)$ is actually a conjugacy class of subgroups of $\Gamma$, so by writing
$\pi_{1}^{\text{lab}}(Y)\leq\pi_{1}^{\text{lab}}(Z)$ we mean that for every $H\in \pi_{1}^{\text{lab}}(Y)$ there is a $J\in \pi_{1}^{\text{lab}}(Z)$ such that $H\leq J$.

This generalizes to $Y,Z$ not necessarily connected in the sense that for every $H$ in a conjugacy class of subgroups $C\in \pi_{1}^{\text{lab}}(Y)$ there is a subgroup $J$ in some
conjugacy class of subgroups $D\in \pi_{1}^{\text{lab}}(Z)$ such that $H\leq J$. We have shown the following observation:

\begin{lem} \label{subgrp lem}
    Let $f:Y\to Z$ be a subcover morphism, where $Y$ and $Z$ are compact. If $\gamma \in \pi_{1}^{\text{lab}}(Y)$ up to conjugation, that is to say that $\gamma$ is an element of a subgroup in a conjugacy class of subgroups in $\pi_{1}^{\text{lab}}(Y)$, then $\gamma \in \pi_{1}^{\text{lab}}(Z)$ up to conjugation.
\end{lem}

Before connecting the previous definitions with our goal of understanding the variables $F_{n}(\gamma)$, we define an asymptotic expansion of a function.
\begin{defn} [Asymptotic expansion]
    Let $k_{1}>k_{2}>...$ and $a_{k_{1}},a_{k_{2}},...$ be sequences of real numbers. We say that a function $f:\mathbb{Z}_{\geq 0}\to\mathbb{R}$ has asymptotic expansion $a_{k_{1}},a_{k_{2}},...$
    if for every $l\in \mathbb{N}$ we have:
    \[
        f(N) = a_{k_{1}}N^{k_{1}}+...+a_{k_{l}}N^{k_{l}} + O\left(N^{k_{l+1}}\right).
    \]
    We denote:
    \[
        f(N) \approx \sum_{l=1}^{\infty}a_{k_{l}}N^{k_{l}}.
    \]
\end{defn}

The connection between the previous few definitions and our goal is given by \cite[Theorem 2.6]{PZ}.
Simply stated, we have an asymptotic expansion of $\mathbb{E}_{Y}^{\text{emb}}(n)$ in terms of the intrinsic properties of $Y$:
\par
\begin{thm} [Asymptotic expansion of $\mathbb{E}_{Y}^{\text{emb}}(n)$] \label{asym exp}
    Let $p:Y\to X$ be a compact, not
    necessarily connected, subcover of $X$. Then there are rational numbers
    $a_{0},a_{-1},a_{-2},...$ such that:
    \[
        \mathbb{E}_{Y}^{\text{emb}}(n)\approx n^{\chi^{\text{grp}}(Y)}\left(a_{0}+a_{-1}n^{-1}+a_{-2}n^{-2}+...\right).
    \]
\end{thm}
\begin{defn} [$a_{0}(Y)$]
    For a compact subcover $p:Y\to X$ we define $a_{0}(Y)$ as the number $a_{0}$
    in the asymptotic expansion of $\mathbb{E}_{Y}^{\text{emb}}(n)$ from Theorem \ref{asym exp}
\end{defn}

Note that throughout the next section, we only use the fact that:
\[
    \mathbb{E}_{Y}^{\text{emb}}(n)= n^{\chi^{\text{grp}}(Y)}\left(a_{0}(Y)+O_{Y}(n^{-1})\right).
\]

Here we conclude the background section of this paper. We are now ready to prove the main theorem - Theorem \ref{main thm} given Lemma \ref{imp lemma} which is stated below.
The proof of Lemma \ref{imp lemma} requires some more background which will be given in Section \ref{lemprf}.

\section{The Proof of Theorem \ref{main thm} and its Corollaries}
\subsection{Proof of The Main Theorem \ref{main thm}}
Before diving into the proof, let us state an important lemma
which will be proven in Section \ref{lemprf}.

\begin{lem} \label{imp lemma}
    Let $E=E_{\gamma}\sqcup E'$ be a compact subcover
    of $X$ such that $\chi^{\text{grp}}(C) = 0$ for all connected components $C$ of $E$.
    If all conjugacy classes in $\pi_{1}^{\text{lab}}(E_{\gamma})$ (or $\pi_{1}^{\text{lab}}(E_{\gamma})$ itself in case $E_{\gamma}$ is connected)
    contain a non-zero power of $\gamma$ (up to conjugation) and all conjugacy classes of $\pi_{1}^{\text{lab}}(E')$
    (or $\pi_{1}^{\text{lab}}(E')$ itself in case $E'$ is connected)
    do not contain any non-zero power of $\gamma$ (up to conjugation),
    then we have \newline $a_{0}(E)=a_{0}(E_{\gamma})a_{0}(E')$.
\end{lem}

\begin{proof} [Proof of Theorem ~\ref{main thm}]
    Let $\gamma_{1},...,\gamma_{t}\in\mathcal{P}_{0}$
    be distinct and for each $i$ let $r_{i}\geq 1$ and $a_{i,1},...,a_{i,r_{i}}\geq1$ be integers.
    Recall that we need to show:
    \begin{equation}\label{31}
        \mathbb{E}_{n}\left[\prod_{i=1}^{t}\prod_{j=1}^{r_{i}}F_{n}(\gamma_{i}^{a_{i,j}})\right]=\prod_{i=1}^{t}\mathbb{E}_{n}\left[\prod_{j=1}^{r_{i}}F_{n}(\gamma_{i}^{a_{i,j}})\right] + O(1/n).
    \end{equation}
    with the implied constant depending on  $\gamma_{1},...,\gamma_{t}$ and the integers $a_{i,j}$.

    Let $R$ be the natural resolution (Definition \ref{nat res}) of:
    \[
        p:Y\stackrel{\text{def}}{=}\bigsqcup_{i=1}^{t}Y_{\gamma_{i}^{a_{i,1}}}\sqcup...\sqcup Y_{\gamma_{i}^{a_{i,r_{i}}}}\to X,
    \]
    and let $R^{0}$ be the subset of $R$ consisting of all the morphisms
    \[
        f:Y\to Z_{f},
    \]
    such that $\chi^{\text{grp}}(Z_{f})=0$.

    Denote
    $p_{ij}\stackrel{\text{def}}{=} p_{\gamma_{i}^{a_{i,j}}}$, i.e.
    $p_{ij}$ is the restricted covering map:
    \[
        p_{ij}:Y_{\gamma_{i}^{a_{i,j}}}\to X.
    \]
    The number of lifts of $p_{ij}$ to $X_{\phi}$ for $\phi\in\text{Hom}(\Gamma,S_{n})$
    is equal $\#\text{Fix}\left(\phi\left(\gamma_{i}^{a_{i,j}}\right)\right)$, hence for a uniformly random $\phi$ we have:

    \[
        \prod_{i,j}F_{n}(\gamma_{i}^{a_{i,j}})=\prod_{i,j}\#\text{lifts of }p_{ij}=\#\text{lifts of }p.
    \]
    From Lemma ~\ref{res and exp} we have:

    \[
        \mathbb{E}_{n}\left[\prod_{i,j}F_{n}(\gamma_{i}^{a_{i,j}})\right]=\mathbb{E}_{Y}(n)=\sum_{f\in R}\mathbb{E}_{Z_{f}}^{\text{emb}}(n).
    \]

    Let $f:Y\twoheadrightarrow Z_{f}$ be in $R$.
    From the asymptotic expansion (Theorem ~\ref{asym exp}) we get:
    \[
        \mathbb{E}_{Z_{f}}^{\text{emb}}(n) = n^{\chi^{\text{grp}}(Z_{f})}\left(a_{0}(Z_{f})+O(n^{-1})\right).
    \]

    As $f:Y\twoheadrightarrow Z_{f}$ is surjective, $Z_{f}$ cannot contain any trivial components, that is, components with trivial $\pi_{1}^{\text{lab}}$. This implies
    that the $\pi_{1}^{\text{lab}}$ of every component of $Z_{f}$ is non-trivial and hence has $\chi^{\text{grp}}\leq 0$, so that $\chi^{\text{grp}}(Z_{f})\leq 0$.
    The above is true for any $f\in R$ so that:
    \[
        \mathbb{E}_{n}\left[\prod_{i,j}F_{n}(\gamma_{i}^{a_{i,j}})\right]=\sum_{f\in R}\mathbb{E}_{Z_{f}}^{\text{emb}}(n)=\sum_{f\in R^{0}}a_{0}(Z_{f}) + O(1/n).
    \]

    If $f\in R^{0}$ then as $\chi^{\text{grp}}(Z_{f})=0$ we deduce that all components of $Z_{f}$ have $\chi^{\text{grp}} = 0$.
    This is due to the fact that the components of $Z_{f}$ have non-positive $\chi^{\text{grp}}$ but $\chi^{\text{grp}}$ is additive and $\chi^{\text{grp}}(Z_{f})=0$.

    For $f\in R^{0}$ the map $f:Y\twoheadrightarrow Z_{f}$ is a surjective morphism of subcovers,
    so from Lemma \ref{subgrp lem} every
    connected component of $Z_{f}$ has one of the $\gamma_{i}^{a_{i,j}}$
    in its $\pi_{1}^{\text{lab}}$ (up to conjugation).

    Suppose that some connected component $C$ of $Z_{f}$
    contains $\gamma_{i}^{a_{i,j}}$ and $\gamma_{i'}^{a_{i',j'}}$ in
    its $\pi_{1}^{\text{lab}}$ (up to conjugation) for $i\neq i'$. In this case $\gamma_{i}^{a_{i,j}}$
    and $\gamma_{i'}^{a_{i',j'}}$ are, up to conjugation, in the same Euler characteristic zero subgroup of
    $\Gamma$, namely a subgroup $G$ of $\Gamma$ in the conjugacy class $\pi_{1}^{\text{lab}}(C)$.
    By the classification of subgroups of $\Gamma$ we know that $G$ must
    be isomorphic to $\mathbb{Z}$. As $\gamma_{i}$ and $\gamma_{i'}$ are non-powers
    we get that $\gamma_{i}$ is conjugate to $\gamma_{i'}$ or its inverse.
    But our assumption is that $\gamma_{i}$ and $\gamma_{i'}$ are distinct in $\mathcal{P}_{0}$ which is a contradiction.

    This implies that one can write $Z_{f}=Z_{f,1}\sqcup...\sqcup Z_{f,t}$
    where $Z_{f,i}$ is the sub-subcover of $Z_{f}$ comprised of all
    connected components which have some $\gamma_{i}^{a_{i,j}}$ in their
    $\pi_{1}^{\text{lab}}$ (up to conjugation).

    Let $R_{i}$ be the natural resolution of:
    \[
        p|_{Y_{\gamma_{i}^{a_{i,1}}}\sqcup...\sqcup Y_{\gamma_{i}^{a_{i,r_{i}}}}}:Y_{\gamma_{i}^{a_{i,1}}}\sqcup...\sqcup Y_{\gamma_{i}^{a_{i,r_{i}}}}\to X,
    \]
    and let $R_{i}^{0}$ be the subset of $R_{i}$ comprised of
    all morphisms in $R_{i}$ with codomain which has $\chi^{\text{grp}}$ zero.

    The decomposition of every $Z_{f}$ as $Z_{f,1}\sqcup...\sqcup Z_{f,t}$ for $f\in R^{0}$ implies that the elements
    of $R^{0}$ are obtained by choosing for each $i$ an element $g_{i}\in R_{i}^{0}$,
    then taking a disjoint union to form an element
    \[
        g_{1}\sqcup...\sqcup g_{t}:Y\to Img_{1}\sqcup...\sqcup Img_{t}
    \]
    in $R^{0}$.

    To see this, first let $f\in R^{0}$.
    Note that by the definition of the $Z_{f,i}$ we know that:
    \[
        f^{-1}(Z_{f,i})=Y_{\gamma_{i}^{a_{i,1}}}\sqcup...\sqcup Y_{\gamma_{i}^{a_{i,r_{i}}}},
    \]
    as only these $Y_{\gamma_{i}^{a_{i,j}}}$ have a power of $\gamma_{i}$ in their $\pi_{1}^{\text{lab}}$ up to conjugation.
    Thus:
    \[
        f|_{Y_{\gamma_{i}^{a_{i,1}}}\sqcup...\sqcup Y_{\gamma_{i}^{a_{i,r_{i}}}}}:Y_{\gamma_{i}^{a_{i,1}}}\sqcup...\sqcup Y_{\gamma_{i}^{a_{i,r_{i}}}}\twoheadrightarrow Z_{f,i},
    \]
    is surjective.
    Because all components of $Z_{f}$ have $\chi^{\text{grp}}$ zero we deduce that
    \[
        \chi^{\text{grp}}(Z_{f,1}) = ... = \chi^{\text{grp}}(Z_{f,t}) = 0.
    \]
    Thus
    \[
        f|_{Y_{\gamma_{i}^{a_{i,1}}}\sqcup...\sqcup Y_{\gamma_{i}^{a_{i,r_{i}}}}} \in R^{0}_{i},
    \]
    for all $i$.

    Conversely, notice that each
    morphism formed by a disjoint union of elements from the $R_{i}^{0}$-s
    is trivially surjective, hence in $R$. Such a morphism is also in $R^{0}$ as $\chi^{\text{grp}}$ is additive.

    From Lemma ~\ref{imp lemma} (combined with a simple induction) we have:

    \[
        \begin{split}
            \sum_{f\in R^{0}}a_{0}(Z_{f}) & =\sum_{f\in R^{0}}a_{0}(Z_{f,1}\sqcup...\sqcup Z_{f,t})                                                                 \\
                                          & =\sum_{f\in R^{0}}a_{0}(Z_{f,1})...a_{0}(Z_{f,t})                                                                       \\
                                          & =\prod_{i}\left(\sum_{f\in R_{i}^{0}}a_{0}(Z_{f})\right)                                                                \\
                                          & =\prod_{i}\biggl(\mathbb{E}_{n}\left[F_{n}(\gamma_{i}^{a_{i,1}})...F_{n}(\gamma_{i}^{a_{i,r_{i}}})\right]+O(1/n)\biggr) \\
                                          & =\prod_{i}\mathbb{E}_{n}\left[F_{n}(\gamma_{i}^{a_{i,1}})...F_{n}(\gamma_{i}^{a_{i,r_{i}}})\right] + O(1/n),            \\
        \end{split}
    \]
    therefore:
    \[
        \begin{split}
            \mathbb{E}_{n}\left[\prod_{i,j}F_{n}(\gamma_{i}^{a_{i,j}})\right] & =\sum_{f\in R^{0}}a_{0}(Z_{f}) + O(1/n)                                                                      \\
                                                                              & =\prod_{i}\mathbb{E}_{n}\left[F_{n}(\gamma_{i}^{a_{i,1}})...F_{n}(\gamma_{i}^{a_{i,r_{i}}})\right] + O(1/n),
        \end{split}
    \]
    which is exactly \ref{31}.
\end{proof}

\subsection{Proof of Corollary \ref{main cor}}
Recall the statement in Corollary \ref{main cor}:
\begin{corollary*}
    Let $\gamma_{1},...,\gamma_{t}\in\mathcal{P}_{0}$
    be distinct and for each $i$ let $r_{i}\geq 1$ and
    \[
        a_{i,1},...,a_{i,r_{i}}\geq1,
    \]
    be integers. We have
    \[
        \mathbb{E}_{n}\left[\prod_{i=1}^{t}\prod_{j=1}^{r_{i}}F_{n}(\gamma_{i}^{a_{i,j}})\right]=\prod_{i=1}^{t}\underset{n\to\infty}{\lim}\mathbb{E}_{n}\left[\prod_{j=1}^{r_{i}}F_{n}(\gamma_{i}^{a_{i,j}})\right] + O(1/n),
    \]
    and
    \[
        \underset{n\to\infty}{\lim}\mathbb{E}_{n}\left[\prod_{i=1}^{t}\prod_{j=1}^{r_{i}}F_{n}(\gamma_{i}^{a_{i,j}})\right]=\prod_{i=1}^{t}\underset{n\to\infty}{\lim}\mathbb{E}_{n}\left[\prod_{j=1}^{r_{i}}F_{n}(\gamma_{i}^{a_{i,j}})\right].
    \]
\end{corollary*}
The proof is rather straightforward.
\begin{proof}[Proof of Corollary \ref{main cor}]
    The second equality follows by taking $n\to\infty$ in Theorem \ref{main thm}.
    As for the first equality, in the setting of the previous proof (the proof of Theorem \ref{main thm}) we have
    \[
        \mathbb{E}_{n}\left[\prod_{i,j}F_{n}(\gamma_{i}^{a_{i,j}})\right]=\sum_{f\in R}\mathbb{E}_{Z_{f}}^{\text{emb}}(n)=\sum_{f\in R^{0}}a_{0}(Z_{f}) + O(1/n),
    \]
    and
    \[
        \sum_{f\in R^{0}}a_{0}(Z_{f})=\prod_{i}\left(\sum_{f\in R_{i}^{0}}a_{0}(Z_{f})\right).
    \]
    But
    \[
        \sum_{f\in R_{i}^{0}}a_{0}(Z_{f}) = \underset{n\to\infty}{\lim}\mathbb{E}_{n}\left[F_{n}(\gamma_{i}^{a_{i,1}})...F_{n}(\gamma_{i}^{a_{i,r_{i}}})\right].
    \]
\end{proof}

\subsection{Proof of Corollary \ref{moments cor}}
Before the proof, recall the statement of Corollary \ref{moments cor}
\begin{corollary*}
    In the same setting as Theorem \ref{main thm} for each positive integer $k$ and $1\leq i \leq t$ let $Z_{1/k}^{(i)}$ be
    a Poisson random variable with parameter $1/k$, such that all $Z$-s are independent (in the strong sense, not just pairwise independent). Define:
    \[
        X^{(i)}_{a_{i,1},...,a_{i,r_{i}}} = \prod_{j=1}^{r_{i}}\sum_{k|a_{i,j}}kZ_{1/k}^{(i)},
    \]
    and note that for different $i$ the variables $X^{(i)}_{a_{i,1},...,a_{i,r_{i}}}$ are independent.
    Then the cross moments of:
    \[
        \prod_{j=1}^{r_{1}}F_{n}(\gamma_{1}^{a_{1,j}}),...,\prod_{j=1}^{r_{t}}F_{n}(\gamma_{t}^{a_{t,j}})
    \]
    and of:
    \[
        X^{(1)}_{a_{1,1},...,a_{1,r_{1}}},...,X^{(t)}_{a_{t,1},...,a_{t,r_{t}}}
    \]
    are asymptotically equal. That is, for every $s_{1},...,s_{t}\in \mathbb{Z}_{\geq 1}$ we have:
    \begin{multline*}
        \underset{n\to\infty}{\lim}\mathbb{E}_{n}\left[\left(\prod_{j=1}^{r_{1}}F_{n}(\gamma_{1}^{a_{1,j}})\right)^{s_{1}},...,\left(\prod_{j=1}^{r_{t}}F_{n}(\gamma_{t}^{a_{t,j}})\right)^{s_{t}}\right] = \\
        \mathbb{E}\left[\left(X^{(1)}_{a_{1,1},...,a_{1,r_{1}}}\right)^{s_{1}}\cdot...\cdot \left(X^{(t)}_{a_{t,1},...,a_{t,r_{t}}}\right)^{s_{t}}\right] = \\
        \mathbb{E}\left[\left(X^{(1)}_{a_{1,1},...,a_{1,r_{1}}}\right)^{s_{1}}\right] \cdot...\cdot\mathbb{E}\left[\left(X^{(t)}_{a_{t,1},...,a_{t,r_{t}}}\right)^{s_{t}}\right].
    \end{multline*}
\end{corollary*}

Once more, the proof is quite simple.
\begin{proof}[Proof of Corollary \ref{moments cor}]
    Let $\gamma_{1},...,\gamma_{t}\in \mathcal{P}_{0}$ be distinct, and let $r_{i}\geq 1$ and $a_{i,1},...,a_{i,r_{i}}\geq1$
    be integers.
    Define:
    \[
        X_{n}^{(i)} = \prod_{j=1}^{r_{i}}F_{n}\left(\gamma_{i}^{a_{i,j}}\right).
    \]
    Let $s_{1},...,s_{t}\in\mathbb{Z}_{\geq1}^{t}$. Note that $\left(X_{n}^{(i)}\right)^{s_{i}}$ is the same expression as if we were to multiply $r_{i}$ by $s_{i}$ and for each $1\leq j\leq r_{i}$ take $s_{i}$ copies of $F_{n}\left(\gamma^{a_{i,j}}\right)$ instead of the one in the original $X_{n}^{(i)}$. Due to this fact, we can assume w.l.g. that $s_{1}=...=s_{t}=1$, so it suffices to show that:

    \begin{multline*}
        \underset{n\to\infty}{\lim}\mathbb{E}_{n}\left[\prod_{i=1}^{t}\prod_{j=1}^{r_{i}}F_{n}(\gamma_{i}^{a_{i,j}})\right] =
        \underset{n\to\infty}{\lim}\mathbb{E}_{n}\left[\prod_{i=1}^{t}X_{n}^{(i)}\right] = \\
        \mathbb{E}\left[\prod_{i=1}^{t}X^{(i)}_{a_{i,1},...,a_{i,r_{i}}}\right]=
        \mathbb{E}\left[\prod_{i=1}^{t}\prod_{j=1}^{r_{i}}\sum_{d|a_{i,j}}dZ_{1/d}^{(i)}\right].
    \end{multline*}

    Recall that $F_{n}(\gamma_{i}^{a_{i,j}})=\sum_{d|a_{i,j}}dC_{n,d}(\gamma_{i})$,
    where $C_{n,d}(\gamma_{i})$ is the number of $d$-cycles in $\phi(\gamma_{i})$
    for a random $\phi\in\text{Hom}(\Gamma,S_{n})$. Thus:

    \[
        \begin{split}
            \mathbb{E}_{n}\left[\prod_{j=1}^{r_{i}}F_{n}(\gamma_{i}^{a_{i,j}})\right] & =\mathbb{E}_{n}\left[\prod_{j=1}^{r_{i}}\sum_{d|a_{i,j}}dC_{n,d}(\gamma_{i})\right]                                      \\
                                                                                      & =\mathbb{E}_{n}\left[\sum_{d_{j}|a_{i,j}} d_{1}...d_{r_{i}}C_{n,d_{1}}(\gamma_{i})...C_{n,d_{r_{i}}}(\gamma_{i})\right]  \\
                                                                                      & =\sum_{d_{j}|a_{i,j}} d_{1}...d_{r_{i}}\mathbb{E}_{n}\left[C_{n,d_{1}}(\gamma_{i})...C_{n,d_{r_{i}}}(\gamma_{i})\right].
        \end{split}
    \]

    Using Theorem \ref{cycle thm} we have:
    \[
        \underset{n\to\infty}{\lim}\mathbb{E}_{n}\left[C_{n,d_{1}}(\gamma_{i})...C_{n,d_{r_{i}}}(\gamma_{i})\right] =
        \mathbb{E}\left[Z_{1/d_{1}}...Z_{1/d_{r_{i}}}\right],
    \]
    which gives:
    \[
        \begin{split}
            \underset{n\to\infty}{\lim}\mathbb{E}_{n}\left[\prod_{j=1}^{r_{i}}F_{n}(\gamma_{i}^{a_{i,j}})\right]
             & =\sum_{d_{j}|a_{i,j}} d_{1}...d_{r_{i}}\underset{n\to\infty}{\lim}\mathbb{E}_{n}\left[C_{n,d_{1}}(\gamma_{i})...C_{n,d_{r_{i}}}(\gamma_{i})\right] \\
             & =\sum_{d_{j}|a_{i,j}} d_{1}...d_{r_{i}}\mathbb{E}\left[Z_{1/d_{1}}...Z_{1/d_{r_{i}}}\right]                                                        \\
             & =\mathbb{E}\left[\prod_{j=1}^{r_{i}}\sum_{d|a_{i,j}}dZ_{1/d}\right].
        \end{split}
    \]
    Applying the second equality in Corollary \ref{main cor} we get:
    \[
        \begin{split}
            \underset{n\to\infty}{\lim}\mathbb{E}_{n}\left[\prod_{i=1}^{t}\prod_{j=1}^{r_{i}}F_{n}(\gamma_{i}^{a_{i,j}})\right]
             & =\prod_{i=1}^{t}\underset{n\to\infty}{\lim}\mathbb{E}_{n}\left[\prod_{j=1}^{r_{i}}F_{n}(\gamma_{i}^{a_{i,j}})\right] \\
             & =\prod_{i=1}^{t}\mathbb{E}\left[\prod_{j=1}^{r_{i}}\sum_{d|a_{i,j}}dZ_{1/d}\right]                                   \\
             & =\mathbb{E}\left[\prod_{i=1}^{t}\prod_{j=1}^{r_{i}}\sum_{d|a_{i,j}}dZ^{(i)}_{1/d}\right],
        \end{split}
    \]
    where the $Z^{(i)}_{1/d}$ are as in the statement of the Corollary.
\end{proof}

\section{Proving Theorems \ref{cyc ext} and \ref{dist thm}}\label{prove dist cor}
Recall the statement of Theorem \ref{cyc ext}:
\begin{thm*}
    Let $\gamma_{1},...,\gamma_{t}\in \mathcal{P}_{0}$ be distinct. For each positive integer $k$ and $1\leq i \leq t$ let $Z_{1/k}^{(i)}$ be
    a Poisson random variable with parameter $1/k$, such that all $Z$-s are independent (in the strong sense, not just pairwise independent).
    In addition for each $i$ let $r_{i}\geq 1$ be an integer.
    Then:
    \begin{multline*}
        \left(C_{n,1}(\gamma_{1}),...,C_{n,r_{1}}(\gamma_{1}),...,C_{n,1}(\gamma_{t}),...,C_{n,r_{t}}(\gamma_{t})\right)\\
        \stackrel{\text{dis}}{\rightarrow}
        \left(Z^{(1)}_{1},...,Z^{(1)}_{1/r_{1}},...,Z^{(t)}_{1},...,Z^{(t)}_{1/r_{t}}\right).
    \end{multline*}
    That is, the number of bounded cycles for different elements in $\mathcal{P}_{0}$ converge jointly to their respective Poisson variable. In particular, the variables in the LHS of the above equation are asymptotically independent.
\end{thm*}

Also recall the statement of Theorem \ref{dist thm}:
\begin{thm*}
    In the same setting as Theorem \ref{main thm}, as $n\to\infty$ we have:
    \begin{multline*}
        \left(\prod_{j=1}^{r_{1}}F_{n}(\gamma_{1}^{a_{1,j}}),...,\prod_{j=1}^{r_{t}}F_{n}(\gamma_{t}^{a_{t,j}})\right)
        \stackrel{\text{dis}}{\rightarrow}
        \left(X^{(1)}_{a_{1,1},...,a_{1,r_{1}}},...,X^{(t)}_{a_{t,1},...,a_{t,r_{t}}}\right) = \\
        \left(\prod_{j=1}^{r_{1}}\sum_{k|a_{1,j}}kZ_{1/k}^{(1)},...,\prod_{j=1}^{r_{t}}\sum_{k|a_{t,j}}kZ_{1/k}^{(t)}\right).
    \end{multline*}
    In particular, the variables
    \[
        \prod_{j=1}^{r_{1}}F_{n}(\gamma_{1}^{a_{1,j}}),...,\prod_{j=1}^{r_{t}}F_{n}(\gamma_{t}^{a_{t,j}})
    \]
    are asymptotically independent as $n\to \infty$, and:
    \[
        \prod_{i=1}^{t}\prod_{j=1}^{r_{i}}F_{n}(\gamma_{i}^{a_{i,j}})\stackrel{\text{dis}}{\to}\prod_{i=1}^{t}\prod_{j=1}^{r_{i}}\sum_{k|a_{i,j}}kZ_{1/k}^{(i)}.
    \]
\end{thm*}

After seeing Corollary \ref{moments cor} one is tempted to use the \textit{method of moments} (see Subsection \ref{MoM sec}) to immediately conclude Theorem \ref{dist thm}. However, to do so one would need to prove that the distributions of variables of the form $X^{(i)}_{a_{i,1},...,a_{i,r_{i}}}$ on $\mathbb{N}$ are determined by their moments. Attempting to do so one quickly notices that the moments of the $X_{a_{1},...,a_{r}}$ grow very fast. For example, when $r=3$ and $a_{1}=a_{2}=a_{3}=1$ we have $X_{1,1,1} = Z_{1}^{3}$, whose moments grow like $\left(\frac{n}{\log{n}}\right)^{3n}$ up to an exponential factor.
This poses a problem when trying to apply standard tests for moment-determinacy such as the existence of a moment-generating function in a neighborhood of $0$ (see Theorem \ref{crit}) or Carlman's condition (see \cite[Theorem 2.7]{carlman}).

Therefore our proof of Theorem \ref{dist thm} circumvents this issue by using Corollary \ref{moments cor} to first prove Theorem \ref{cyc ext} using the method of moments, then using Theorem \ref{cyc ext} to prove Theorem \ref{dist thm}.

Before the proof of Theorem \ref{cyc ext}
we give some relevant background on the method of moments.
\subsection{The Method of Moments}\label{MoM sec}
First, we define a measure determined by its moments. The definition comes from \cite[Ex. 30.5]{BP}:
\begin{defn}[Measure Determined by its Moments]
    Let $\mu$ be probability distribution on the space $\mathbb{R}^{m}$. We say that $\mu$ is \textit{determined by its moments} if any other probability measure $\nu$ on $\mathbb{R}^{m}$
    that satisfies:
    \[
        \int_{\mathbb{R}^{m}} x_{1}^{r_{1}}\cdot\cdot\cdot x_{m}^{r_{m}} d\mu = \int_{\mathbb{R}^{m}} x_{1}^{r_{1}}\cdot\cdot\cdot x_{m}^{r_{m}} d\nu,
    \]
    for all positive integers $r_{1},...,r_{m}$ is equal to $\mu$. That is, $\mu$ is the only probability measure on $\mathbb{R}^{m}$ which has moments:
    \[
        \int_{\mathbb{R}^{m}} x_{1}^{r_{1}}\cdot\cdot\cdot x_{m}^{r_{m}} d\mu.
    \]
\end{defn}

Given a probability distribution on $\mathbb{R}^{m}$, the problem of deciding whether it is determined by its moments has been well studied.
A simple test for moment-determinacy in the $m=1$ case is the following:
\begin{thm}[Theorem 30.1 in \cite{BP}]\label{crit}
    Let $\mu$ be probability distribution on $\mathbb{R}$ with finite moments $m_{k} = \int_{\mathbb{R}}x^{k}d\mu(x)$. If the moment-generating function of $\mu$ exists in a neighborhood of $0$ then
    $\mu$ is determined by its moments. Restated, if there is a positive $r$ for which the series $\sum_{k\geq 0} m_{k}r^{k}/k!$ converges, then $\mu$ is determined by its moments.
\end{thm}

Note that a Poisson random variable with parameter $\lambda$ has a moment generating function equal to $e^{\lambda(e^{r}-1)}$, see \cite[Equation 21.27]{BP}.
In particular, the series $\sum_{k\geq 0} m_{k}r^{k}/k!$ converges for all $r$ and thus the distributions of Poisson variables are determined by their moments.

Finally, the so-called "method of moments" in the multivariate case is the next theorem.
We use \cite[Ex. 30.6]{BP}, however we state the theorem as it appears in \cite[Theorem 6.1]{PZ}:
\begin{thm}[Multivariate Method of Moments]\label{moments}
    Let $X^{(1)},...,X^{(m)}$ and \newline $X_{n}^{(1)},...,X_{n}^{(m)}$ be random variables, and suppose that the joint distribution of $X^{(1)},...,X^{(m)}$ is determined by its moments. If the variables $X_{n}^{(1)},...,X_{n}^{(m)}$ have moments of all orders for all $n$ such that
    \[
        \underset{n\to\infty}{\lim}\mathbb{E}\left[\left(X_{n}^{(1)}\right)^{s_{1}}\cdot\cdot\cdot\left(X_{n}^{(m)}\right)^{s_{m}}\right] =
        \mathbb{E}\left[\left(X^{(1)}\right)^{s_{1}}\cdot\cdot\cdot\left(X^{(m)}\right)^{s_{m}}\right],
    \]
    for all $s_{1},...,s_{m}\in \mathbb{N}^{m}$, then
    \[
        \left(X_{n}^{(1)},...,X_{n}^{(m)}\right) \xrightarrow{dis} \left(X^{(1)},...,X^{(m)}\right).
    \]
    In particular, if $X^{(1)},...,X^{(m)}$ are independent then $X_{n}^{(1)},...,X_{n}^{(m)}$ are asymptotically independent.
\end{thm}

The useful proposition below lets one construct multivariate probability measures that are determined by their moments using univariate probability measures that are determined by their moments.
\begin{prop}[Theorem 3 in \cite{moments}]\label{s obs}
    Let $\mu_{1},...,\mu_{m}$ be probability  distributions on $\mathbb{R}$ that are determined by their moments. Then the product probability distribution
    \[
        \mu_{1}\times ...\times \mu_{m}
    \]
    on $\mathbb{R}^{m}$ is determined by its moments.
\end{prop}

Applying Proposition \ref{s obs} to our setting, we conclude:
\begin{prop}\label{pos det mom}
    Let $\lambda_{1},...,\lambda_{m}$ be positive numbers, and let $Z_{\lambda_{1}},...,Z_{\lambda_{m}}$ be independent Possion random variables. The joint distribution of:
    \[
        \left(Z_{\lambda_{1}},...,Z_{\lambda_{m}}\right),
    \]
    on $\mathbb{R}^{m}$ is determined by its moments.
\end{prop}

\subsection{Proof of Theorem \ref{cyc ext} and of Theorem \ref{dist thm}}
Before the proof, let us recall some simple machinery from number theory.
\begin{defn}[Definition 16.3 in \cite{HW}]
    The M\"{o}bius function:
    \[
        \mu:\mathbb{Z}_{\geq 1}\to \{0,\pm1\},
    \]
    is defined as follows:
    \begin{enumerate}
        \item
              $\mu(1) = 1$
        \item
              $\mu(n) = 0$ if $n$ has a square factor.
        \item
              $\mu(n) = (-1)^k$ if $n$ is square-free and is divisible by exactly $k$ different primes.
    \end{enumerate}
\end{defn}

\begin{thm}[The M\"{o}bius Inversion Formula - Theorem 16.4 in \cite{HW}]\label{mob inv}
    Let \newline $f,g:\mathbb{Z}_{\geq 1} \to \mathbb{R}$. If:
    \[
        g(n) = \sum_{d|n}f(d),
    \]
    for all $n$, then:
    \[
        f(n) = \sum_{d|n}\mu(n/d)g(d) = \sum_{d|n}\mu(d)g(n/d),
    \]
    for all $n$.
\end{thm}

\begin{proof}[Proof of Theorem \ref{cyc ext}]
    Let $\gamma_{1},...,\gamma_{t}\in \mathcal{P}_{0}$ be distinct, and for each $i$ let $r_{i}\geq 1$ be an integer.
    To show that:
    \begin{multline*}
        \left(C_{n,1}(\gamma_{1}),...,C_{n,r_{1}}(\gamma_{1}),...,C_{n,1}(\gamma_{t}),...,C_{n,r_{t}}(\gamma_{t})\right)\\
        \stackrel{\text{dis}}{\rightarrow}
        \left(Z^{(1)}_{1},...,Z^{(1)}_{1/r_{1}},...,Z^{(t)}_{1},...,Z^{(t)}_{1/r_{t}}\right),
    \end{multline*}
    holds, we use Theorem \ref{moments} alongside Proposition \ref{pos det mom}. It thus suffices to show that for every choice of positive integers $s_{i,j}$ for $1\leq i\leq t$ and $1\leq j\leq r_{i}$ we have:
    \[
        \underset{n\to\infty}{\lim}\mathbb{E}_{n}\left[\prod_{i=1}^{t}\prod_{j=1}^{r_{i}}C_{n,j}(\gamma_{i})^{s_{i,j}}\right] =
        \mathbb{E}\left[\prod_{i=1}^{t}\prod_{j=1}^{r_{i}}\left(Z^{(i)}_{1/j}\right)^{s_{i,j}}\right].
    \]

    Using the M\"{o}bius inversion formula, Theorem \ref{mob inv}, alongside the identity $F_{n}(\gamma^{q}) = \sum_{d|q}dC_{n,d}(\gamma)$ yields:
    \[
        C_{n,r}(\gamma_{i}) = \frac{1}{r}\sum_{d|r}\mu(d)F_{n}(\gamma_{i}^{r/d}).
    \]
    Therefore it suffices to show that for every choice of positive integers $s_{i,j}$ for $1\leq i\leq t$ and $1\leq j\leq r_{i}$ we have:
    \[
        \underset{n\to\infty}{\lim}\mathbb{E}_{n}\left[\prod_{i=1}^{t}\prod_{j=1}^{r_{i}}\left(\frac{1}{j}\sum_{d|j}\mu(d)F_{n}(\gamma_{i}^{j/d})\right)^{s_{i,j}}\right] =
        \mathbb{E}\left[\prod_{i=1}^{t}\prod_{j=1}^{r_{i}}\left(Z^{(i)}_{1/j}\right)^{s_{i,j}}\right].
    \]
    However, as the LHS is a polynomial in the variables:
    \[
        \{F_{n}(\gamma_{i}^{a})\}_{\substack{1\leq i\leq t\\1\leq a\leq \max_{i}r_{i}}},
    \]
    Corollary \ref{moments cor} shows that in the limit as $n\to\infty$ we may swap each $F_{n}(\gamma_{i}^{a})$ by $X^{(i)}_{a} = \sum_{d|a}dZ^{(i)}_{1/d}$. Formally:
    \begin{multline*}
        \underset{n\to\infty}{\lim}\mathbb{E}_{n}\left[\prod_{i=1}^{t}\prod_{j=1}^{r_{i}}\left(\frac{1}{j}\sum_{d|j}\mu(d)F_{n}(\gamma_{i}^{j/d})\right)^{s_{i,j}}\right] =\\
        \mathbb{E}\left[\prod_{i=1}^{t}\prod_{j=1}^{r_{i}}\left(\frac{1}{j}\sum_{d|j}\mu(d)X^{(i)}_{j/d}\right)^{s_{i,j}}\right].
    \end{multline*}
    As $X^{(i)}_{a} = \sum_{d|a}dZ^{(i)}_{1/d}$, the M\"{o}bius inversion formula, Theorem \ref{mob inv}, shows that:
    \[
        Z^{(i)}_{1/j} = \frac{1}{j}\sum_{d|j}\mu(d)X^{(i)}_{j/d}.
    \]
    In total:
    \[
        \begin{split}
            \underset{n\to\infty}{\lim}\mathbb{E}_{n}\left[\prod_{i=1}^{t}\prod_{j=1}^{r_{i}}C_{n,j}(\gamma_{i})^{s_{i,j}}\right] & =
            \underset{n\to\infty}{\lim}\mathbb{E}_{n}\left[\prod_{i=1}^{t}\prod_{j=1}^{r_{i}}\left(\frac{1}{j}\sum_{d|j}\mu(d)F_{n}(\gamma_{i}^{j/d})\right)^{s_{i,j}}\right]                                                                                  \\
                                                                                                                                  & = \mathbb{E}\left[\prod_{i=1}^{t}\prod_{j=1}^{r_{i}}\left(\frac{1}{j}\sum_{d|j}\mu(d)X^{(i)}_{j/d}\right)^{s_{i,j}}\right] \\
                                                                                                                                  & =\mathbb{E}\left[\prod_{i=1}^{t}\prod_{j=1}^{r_{i}}\left(Z^{(i)}_{1/j}\right)^{s_{i,j}}\right].
        \end{split}
    \]
\end{proof}

\begin{proof}[Proof of Theorem \ref{dist thm}]
    Given Theorem \ref{cyc ext} the proof is very short. Let \newline $\gamma_{1},...,\gamma_{t}\in \mathcal{P}_{0}$ be distinct and let $r_{i}$ and $a_{i,1},...,a_{i,r_{i}}$ be positive integers. Define $g:\mathbb{R}^{\sum_{i}r_{i}}\to \mathbb{R}^{t}$ as:
    \[
        g(x_{1,1},...,x_{1,r_{1}},...,x_{t,1},...,x_{t,r_{t}}) = \left(\prod_{j=1}^{r_{1}}\sum_{k|a_{1,j}}kx_{1,k},...,\prod_{j=1}^{r_{t}}\sum_{k|a_{t,j}}kx_{t,k}\right),
    \]
    and note that it is continuous.

    As $F_{n}(\gamma^{a}) = \sum_{d|a}dC_{n,d}(\gamma)$ for every $\gamma \in \Gamma$ and positive integer $a$, we have:
    \begin{multline*}
        g\left(C_{n,1}(\gamma_{1}),...,C_{n,r_{1}}(\gamma_{1}),...,C_{n,1}(\gamma_{t}),...,C_{n,r_{t}}(\gamma_{t})\right) =\\
        \left(\prod_{j=1}^{r_{1}}F_{n}(\gamma_{1}^{a_{1,j}}),...,\prod_{j=1}^{r_{t}}F_{n}(\gamma_{t}^{a_{t,j}})\right),
    \end{multline*}
    and:
    \begin{multline*}
        g\left(Z^{(1)}_{1},...,Z^{(1)}_{1/r_{1}},...,Z^{(t)}_{1},...,Z^{(t)}_{1/r_{t}}\right) = \\
        \left(\prod_{j=1}^{r_{1}}\sum_{k|a_{1,j}}kZ_{1/k}^{(1)},...,\prod_{j=1}^{r_{t}}\sum_{k|a_{t,j}}kZ_{1/k}^{(t)}\right) =
        \left(X^{(1)}_{a_{1,1},...,a_{1,r_{1}}},...,X^{(t)}_{a_{t,1},...,a_{t,r_{t}}}\right).
    \end{multline*}
    Theorem \ref{cyc ext} alongside the fact that $g$ is continuous gives:
    \[
        \left(\prod_{j=1}^{r_{1}}F_{n}(\gamma_{1}^{a_{1,j}}),...,\prod_{j=1}^{r_{t}}F_{n}(\gamma_{t}^{a_{t,j}})\right)
        \stackrel{\text{dis}}{\rightarrow}
        \left(X^{(1)}_{a_{1,1},...,a_{1,r_{1}}},...,X^{(t)}_{a_{t,1},...,a_{t,r_{t}}}\right).
    \]
\end{proof}

\section{Proving Lemma \ref{imp lemma}}\label{lemprf}
For the rest of this paper, we focus on Lemma \ref{imp lemma}.
As this is the most technical section of the paper, the reader is highly recommended to read the outline of the proof below.
Before continuing, let us restate the lemma here:
\begin{lemma*}
    Let $E=E_{\gamma}\sqcup E'$ be a compact subcover
    of $X$ such that $\chi^{\text{grp}}(C) = 0$ for all connected components $C$ of $E$.
    If all conjugacy classes in $\pi_{1}^{\text{lab}}(E_{\gamma})$ (or $\pi_{1}^{\text{lab}}(E_{\gamma})$ itself in case $E_{\gamma}$ is connected)
    contain a non-zero power of $\gamma$ (up to conjugation) and all conjugacy classes of $\pi_{1}^{\text{lab}}(E')$
    (or $\pi_{1}^{\text{lab}}(E')$ itself in case $E'$ is connected)
    do not contain any non-zero power of $\gamma$ (up to conjugation),
    then we have \newline $a_{0}(E)=a_{0}(E_{\gamma})a_{0}(E')$.
\end{lemma*}

\subsection{Outline of the proof}
To prove Lemma ~\ref{imp lemma} we need a more direct understanding of the number $a_{0}(Y)$ for subcovers $Y$ of $X$.
As it turns out, computing $a_{0}(Y)$ boils down to counting the number of subcovers with a special property in some special resolution of $Y$.
A subcover having this special property is called "strongly boundary reduced" or SBR for short.
The reason we care about these 'special' subcovers is that for SBR $Z$ we have $a_{0}(Z)=1$.
In the proof to come we do not need an exact definition of when a subcover is SBR,
just the fact that a subcover is SBR if and only if all of its components are. The proof will consist of the following three steps:
\begin{enumerate}
    \item
          Fix a special embedding-resolution $R^{\text{emb}}_{E}$ of $E=E_{\gamma} \sqcup E'$ with the property that if
          $h:E\hookrightarrow Z_{h}$ is in $R^{\text{emb}}_{E}$ such that $\chi^{\text{grp}}(Z_{h}) = 0$ then $Z_{h}$ is SBR.
    \item
          Use Lemma \ref{res and exp} to show that
          \[
              a_{0}(E) = \#\{h\in R^{\text{emb}}_{E} : \chi^{\text{grp}}(Z_{h})=0\}.
          \]
    \item
          Set  $R^{\text{emb}}_{E_{\gamma}}$ and $R^{\text{emb}}_{E'}$ to be the special embedding-resolutions of $E_{\gamma}$ and $E'$, respecively.
          Finally, we show that elements in
          \[
              \{h\in R^{\text{emb}}_{E} : \chi^{\text{grp}}(Z_{h})=0\},
          \]
          are disjoint unions of morphisms $f\in R^{\text{emb}}_{E_{\gamma}}$
          and of a $g\in R^{\text{emb}}_{E'}$, such that:
          \[
              f:E_{\gamma}\hookrightarrow W_{\gamma} \;\;\;\;\; g:E'\hookrightarrow W'
          \]
          and $\chi^{\text{grp}}(W_{\gamma}) = \chi^{\text{grp}}(W')=0$,
          and that each such disjoint union appears in
          \[
              \{h\in R^{\text{emb}}_{E}:\chi^{\text{grp}}(Z_{h})=0\}.
          \]
          This concludes the proof.
\end{enumerate}

\subsection{Additional Background}
\subsubsection{background for step 1 of the proof}
We now give a brief recap of the terms to be used in the proof, a great reference that elaborates further is \cite{MP2}.
Recall that $X$ has a CW structure consisting of one vertex $o$, $2g$
1-cells (edges) and a single $2$-cell, and that every cover of $X$
inherits the CW structure from $X$. As covers of $X$ are surfaces
themselves, for a subcover $p:Y\to X$ where $Y\subseteq Z$ and
$Z$ is a cover of $X$, we can take a small
closed regular neighborhood of $Y$ in $Z$ and obtain the \textbf{thick
    version} of $Y$ denoted $\mathbb{Y}$. Note that the thick version
of a subcover is also a surface, possibly with boundary, and does
not depend on the cover $Z$ in which $Y$ is embedded, rather it
is a feature of $Y$ \cite[Section 3]{MP2}.

Denote by $\partial Y$ the boundary of
$\mathbb{Y}$ which, for compact $Y$, is just a finite collection of cycles. On each of the cycles
we can choose an orientation to obtain a \textbf{boundary cycle} of
$Y$. This boundary cycle corresponds to a (cyclic) word in the generators
of $\Gamma$ by the edge labels along the boundary cycle. Note that
every cover $Z$ of $X$ consists of vertices, directed edges
labeled by $a_{1},b_{1},...,a_{g},b_{g}$ and $4g$-gons, where the
cycle around every $4g$-gon reads the relation $[a_{1},b_{1}]...[a_{g},b_{g}]$.
We choose orientations on boundary cycles of $Y$ such that if $Y$
is embedded in $Z$, the boundary cycles read successive segments
of neighboring $4g$-gons (in $Z\backslash Y$) with the orientation
from the relation $[a_{1},b_{1}]...[a_{g},b_{g}]$ and not the inverse
word.

It might be the case that a subcover $p:Y\to X$ has a boundary
cycle that contains a sub-word of $[a_{1},b_{1}]...[a_{g},b_{g}]$
of length $\geq2g+1$, in that case, one can shorten the total boundary
length of $Y$ by annexing the $4g$-gon neighboring this sub-word,
hence $\partial Y$ is ``not reduced''. We call such a sub-word
of $[a_{1},b_{1}]...[a_{g},b_{g}]$ a \textbf{long block}. There are
more intricate cases where the boundary of $Y$ is ``not reduced''. Specifically, when
it contains certain words on $\{a_{1}^{\pm1},b_{1}^{\pm1},...,a_{g}^{\pm1},b_{g}^{\pm1}\}$
which we call \textbf{long chains}. The precise definition of long
chains can be found in \cite[Section 3]{MP2}.

\begin{defn} [Boundary reduced subcovers] A subcover $Y$ of $X$
    is called \textbf{boundary reduced }(BR) if $\partial Y$ contains
    neither long blocks nor long chains.
    \par
\end{defn}

In particular, from the alternative definition of a sub-cover $Y$, Definition \ref{alt def},
if $Y$ is BR then every path in $Y$ reading the relation $[a_{1},b_{1}]...[a_{g},b_{g}]$
must be the boundary of a $4g$-gon in $Y$.
\par
We do not need the explicit definition of a long chain, just the notion
of BR subcovers and the fact that a subcover $Y$ being BR only depends
on $\partial Y$.
\par
Given a subcover $Y$ of $X$ it is possible for $Y$ to have boundary
cycles which contain half of the relation $[a_{1},b_{1}]...[a_{g},b_{g}]$,
i.e. a cyclic sub-word of \newline $[a_{1},b_{1}]...[a_{g},b_{g}]$ of length $2g$.
In that case, one can annex the $4g$-gon neighboring this word (in
$Z\backslash Y$ where $Y\subseteq Z$ with $Z$ a full cover of $X$)
without increasing the total length of $\partial Y$, even possibly shortening it. Such a sub-word
is called a \textbf{half block}. Similarly, one has a notion of\textbf{
    half chains} whose precise definition can be found in \cite[Section 3]{MP2}.
\par

\begin{defn} [Strongly boundary reduced subcovers] A subcover $Y$
    of $X$ is called \textbf{strongly boundary reduced }(SBR) if $\partial Y$
    contains neither half blocks nor half chains.
    \par
\end{defn}

Note that as the naming suggests, an SBR subcover is also BR. See \newline \cite[Section 4]{MP2}. Again, we do not need the precise definition of half chains,
we only need the fact that a compact subcover being SBR only depends on its
boundary cycles. In addition, we need the following fact which is \newline \cite[Theorem 3.11, part 2]{PZ} and can also be found in \cite{MP1}:

\begin{thm} \label{a zero one}
    Let $Y$ be a compact SBR subcover of $X$, then:
    \[
        a_{0}(Y)=1.
    \]
\end{thm}

The first step of our proof requires the existence of special resolutions for compact subcovers. Their existence comes from the following theorem
and a corollary of it.

\begin{thm}\cite[Theorem 3.12]{PZ} \label{sp res}
    Let $Y$ be a compact subcover
    of $X$ and let $\chi_{0}\in\mathbb{Z}$. Then $Y$ admits a finite
    resolution $R_{Y}$ such that for every $f:Y\to Z_{f}$ in $R_{Y}$ we
    have
    \begin{itemize}
        \item
              the subcover $Z_{f}$ is compact and BR,
        \item
              if $\chi^{\text{grp}}(Z_{f})\geq\chi_{0}$ then $Z_{f}$ is SBR, and
        \item
              the image of $f$ meets every connected component of $Z_{f}$.
              \par
    \end{itemize}
\end{thm}

The corollary is:
\begin{cor} \label{sp emb res}
    Let $Y$ be a compact subcover
    of $X$ and let $\chi_{0}\in\mathbb{Z}$. Then $Y$ admits a finite
    embedding-resolution $R^{\text{emb}}_{Y}$ with the same three properties as in Theorem \ref{sp res}.
\end{cor}

\subsubsection{Background for step 2 of the proof}
Theorem ~\ref{sp res} is stated for any $\chi_{0}$, but as in step 2 in the outline of the proof, the $\chi_{0}=0$
case is the most interesting for us as a result of the following statement:

\begin{prop} \label{imp prop} For a compact subcover $Y$ of $X$ with $\chi^{\text{grp}}(Y)=0$ we have:
    \[
        a_{0}(Y)=\#\{f\in R^{\text{emb}}_{Y}|\chi^{\text{grp}}(Z_{f})=0\},
    \]
    where $R^{\text{emb}}_{Y}$ is the finite resolution from Corollary
    ~\ref{sp emb res} when $\chi_{0}=0$.
    \par
\end{prop}

To prove Proposition \ref{imp prop} we need the following lemma which gives an important relation between subcover embeddings and $\chi^{\text{grp}}$.
\begin{lem} [Lemma 3.14 of \cite{PZ}] \label{EC lem}
    If $f:Y\hookrightarrow Z$ is an embedding of compact subcovers of $X$ such that the image of $f$ meets every connected component of $Z$, then $\chi^{\text{grp}}(Z)\leq \chi^{\text{grp}}(Y)$
\end{lem}

\begin{proof}[Proof of Proposition \ref{imp prop}]
    Let $R^{\text{emb}}_{Y}$ be the finite embedding-resolution from Corollary ~\ref{sp emb res} when $\chi_{0}=0$. From Lemma \ref{EC lem} and the third bulletin of
    Theorem \ref{sp res} we know that $\chi^{\text{grp}}(Z_{f})\leq \chi^{\text{grp}}(Y)=0$ for all $f:Y\hookrightarrow Z_{f}$ in $R^{\text{emb}}_{Y}$.
    Using the asymptotic expansion (Theorem \ref{asym exp}) and Lemma \ref{res and exp} we have:

    \[
        \begin{split}
            a_{0}(Y)+O(1/n) & = \sum_{f\in R^{\text{emb}}_{Y}} \mathbb{E}^{\text{emb}}_{Z_{f}}(n)                          \\
                            & =\sum_{f\in R^{\text{emb}}_{Y}} n^{\chi^{\text{grp}}(Z_{f})}\left[a_{0}(Z_{f})+O(1/n)\right] \\
                            & = \sum_{f\in R^{\text{emb}}_{Y} : \chi^{\text{grp}}(Z_{f}) = 0} a_{0}(Z_{f}) +O(1/n)         \\
                            & = \#\{f\in R^{\text{emb}}_{Y} : \chi^{\text{grp}}(Z_{f})=0\} +O(1/n),
        \end{split}
    \]
    with the last equality following from Theorem \ref{a zero one} and the second bulletin of Theorem \ref{sp res}. Letting $n\to \infty$ we are done.

\end{proof}

\subsubsection{Background for step 3 of the proof}
To complete step 3 as in the outline we need an explicit form of the embedding-resolution as in Corollary \ref{sp emb res}.
As we shall see, for a compact subcover $Y$ of $X$ the elements of the special resolution in Theorem \ref{sp res} are outputs of a certain algorithm which Magee and Puder dub 'The Growing Process' in \cite{MP1}.
The algorithm takes as input a subcover morphism $f:Y\to Z$ where $Y$ is compact and $Z$ is without boundary, that is, $Z$ is a full cover of $X$, and spits out a morphism $g:Y\to W$ for $W\subseteq Z$ such that either $W$ is SBR, or $W$ is BR with $\chi(W)<0$ ($\chi$ being the normal topological Euler Characteristic).
The algorithm is as follows:
\begin{algorithm}[H]
    \textbf{Growing Process}:
    \par
    For a subcover morphism $f:Y\to Z$, with $Y$ compact and $Z$ without
    boundary:
    \par
    \begin{enumerate}
        \item
              Initialize $W_{0}=Im(f)$ and $i=0$.
              \par
        \item
              If $W_{i}$ is SBR, or is BR with $\chi(W_{i})<0$, terminate and return \newline
              $g:Y\to W_{i}$.
              \par
        \item \raggedright{}Obtain $W_{i+1}$ from $W_{i}$ by adding to $W_{i}$
              (the closure of) every $4g$-gon in $Z\backslash W_{i}$ which touches
              along its boundary an edge of $\partial W_{i}$ which is part of a
              half-block (this includes the case of a long block), a long-chain
              or a half-chain. Set $i = i+1$ and return to 2.
    \end{enumerate}
\end{algorithm}

Lemma 2.11 of \cite{MP1} states that this process always terminates.
With this terminology in mind, the special resolution $R_{Y}$ of a compact subcover $Y$ as in Theorem \ref{sp res} with $\chi_{0} = 0$ is exactly all possible outputs of the growing process.
Note that if $Y$ is compact then there are finitely many possible outputs of the growing process, see \cite[Lemma 2.12]{MP1}.
In addition, the special embedding-resolution $R^{\text{emb}}_{Y}$ of a compact subcover $Y$ as in Corollary \ref{sp emb res} with $\chi_{0} = 0$ is the set of all possible $\textbf{injective}$ outputs of the growing process.

For a subcover morphism $h:Y\to Q$ where $Y$ is compact and $Q$ without boundary,
we set $W_{i}(h)$ to be the subcover $W_{i}$ in the $i$'th iteration of the growing process on $h$.
We also denote by $i_{\text{final}}(h)$ the $i$ on which the growing process on $h$ terminated. In particular $W_{i_{\text{final}}(h)}(h)$ is always SBR or
is BR with $\chi\left(W_{i_{\text{final}}(h)}(h)\right) < 0$.
Before diving into the proof we make a key observation that enables our proof to work.

\begin{prop}[Key Observation]\label{key obs}
    Let $Y$ be a compact subcover and $Q$ without boundary. In addition let $h:Y\hookrightarrow Q$ be an embedding.
    Suppose that there exists a compact $Z\subseteq Q$ such that $W_{i_{\text{final}}(h)}(h)\subseteq Z$.
    Also assume that $\chi^{\text{grp}}(Y) = 0$, that $\chi^{\text{grp}}(Z)=0$ and that $Im(h)$ meets all of $Z$'s components.
    Then for all $i\leq i_{\text{final}}(h)$ we have:
    \[
        \chi^{\text{grp}}(W_{i}(h))=0.
    \]
    In particular, $W_{i_{\text{final}}(h)}(h)$ is SBR.
\end{prop}
\begin{proof}
    The proof is rather simple. Note that:
    \[
        Y\hookrightarrow Im(h)=W_{0}(h)\subseteq W_{1}(h)\subseteq...\subseteq W_{i_{\text{final}}(h)}(h) \subseteq Z.
    \]
    In addition, note that $Y$ being compact implies that all $W_{i}(h)$ are compact as well. Also, by definition of the growing process, it is easy to see that $Im(h)$ meets all components of  $W_{i}(h)$ for all $i$. Lemma \ref{EC lem} implies that
    \[
        \chi^{\text{grp}}(W_{i}(h)) \leq \chi^{\text{grp}}(Y) = 0,
    \]
    for all $i$.

    On the other hand, as $Im(h)$ meets all components of $Z$, for all $i$ there is a morphism which meets all components of $Z$ given by:
    \[
        W_{i}(h)\hookrightarrow W_{i_{\text{final}}(h)}(h) \hookrightarrow Z.
    \]
    Lemma \ref{EC lem} yields
    \[
        0 = \chi^{\text{grp}}(Z) \leq \chi^{\text{grp}}(W_{i}(h)).
    \]
    Combining the two inequalities we see that $\chi^{\text{grp}}(W_{i}(h)) = 0$ for all $i$, and in particular $\chi^{\text{grp}}(W_{i_{\text{final}}(h)}(h)) = 0$.

    From \cite[paragraph following Theorem 3.12]{PZ} we know that $\chi = \chi^{\text{grp}}$ for BR subcovers, which implies $\chi \left(W_{i_{\text{final}}(h)}(h)\right) = 0$.
    By definition, the growing process can not terminate when $W_{i_{\text{final}}(h)}(h)$ is BR and $\chi \left(W_{i_{\text{final}}(h)}(h)\right) = 0$ which means that  $W_{i_{\text{final}}(h)}(h)$ is SBR.
\end{proof}

\subsection{The Proof}
Let $R^{\text{emb}}_{E}$ be the embedding-resolution of $E$ as in Corollary \ref{sp emb res} with $\chi_{0} = 0$, and similarly let $R^{\text{emb}}_{E_{\gamma}}$ and $R^{\text{emb}}_{E'}$ be the special embedding-resolutions of $E_{\gamma},E'$ as in Corollary \ref{sp emb res} with $\chi_{0} = 0$ respectively. Denote by $R^{\text{emb}}_{E}(0)$ the subset $\{h\in R^{\text{emb}}_{E} | \chi^{\text{grp}}(Z_{h})=0\}$ of $R^{\text{emb}}_{E}$, and similarly define $R^{\text{emb}}_{E_{\gamma}}(0)$ and $R^{\text{emb}}_{E'}(0)$.
Proposition \ref{imp prop} gives
\[
    a_{0}(E)=\#R^{\text{emb}}_{E}(0)\;\;\;a_{0}(E_{\gamma}) = \#R^{\text{emb}}_{E_{\gamma}}(0) \;\;\;a_{0}(E') = \#R^{\text{emb}}_{E'}(0).
\]
Thus, it suffices to show that:
\begin{enumerate}
    \item[\textbf{Part 1}]
          Every disjoint union $f\sqcup g$ for $f\in R^{\text{emb}}_{E_{\gamma}}(0)$ and $g\in R^{\text{emb}}_{E'}(0)$ belongs to $R^{\text{emb}}_{E}(0)$.
    \item[\textbf{Part 2}]
          Every $h\in R^{\text{emb}}_{E}(0)$ can be written as $h=f\sqcup g$ for $f\in R^{\text{emb}}_{E_{\gamma}}(0)$ and $g\in R^{\text{emb}}_{E'}(0)$.
\end{enumerate}

\begin{proof}[\textbf{Part 1}]
    Let $f\in R^{\text{emb}}_{E_{\gamma}}(0)$ and  $g\in R^{\text{emb}}_{E'}(0)$ and write $f:E_{\gamma} \hookrightarrow Z_{f}$ and \newline$g:E'\hookrightarrow Z_{g}$.
    By definition of $R^{\text{emb}}_{E_{\gamma}}$ and $R^{\text{emb}}_{E'}$ there are morphisms \newline$h_{f}:E_{\gamma} \hookrightarrow Q_{\gamma}$ and $h_{g}:E' \hookrightarrow Q'$ such that $Q_{\gamma}$ and $Q'$ are without boundary and such that $f$ and $g$ are the output of the growing process on $h_{f}$ and $h_{g}$ respectively.

    Consider the morphism $h_{f\sqcup g}$ defined as
    \[
        h_{f}\sqcup h_{g}:E = E_{\gamma}\sqcup E' \hookrightarrow Q \stackrel{def}{=} Q_{\gamma}\sqcup Q'.
    \]
    Let us examine what happens when we run the growing process on $h_{f\sqcup g}$.

    Suppose that $i_{\text{final}}(h_{f})\leq i_{\text{final}}(h_{g})$. Because the process is deterministic, this implies that for all $i\leq i_{\text{final}}(h_{f})$ we have:
    \[
        W_{i}(h_{f\sqcup g}) = W_{i}(h_{f})\sqcup W_{i}(h_{g}).
    \]
    Let us consider what happens at time $i_{\text{final}}(h_{f})$. By definition of $h_{f}$, we know that at time $i_{\text{final}}(h_{f})$ the subcover $W_{i_{\text{final}}(h_{f})}(h_{f}) = Z_{f}$ is SBR. This implies that there are no half-blocks, half-chains, or long chains at the boundary of $W_{i_{\text{final}}(h_{f})}(h_{f})$. Thus, as the process continues, step 3 does not change $W_{i_{\text{final}}(h_{f})}(h_{f}) = Z_{f}$ so that for $i\geq i_{\text{final}}(h_{f})$
    \[
        W_{i}(h_{f\sqcup g}) = W_{i_{\text{final}}(h_{f})}(h_{f})\sqcup W_{i}(h_{g}) = Z_{f} \sqcup W_{i}(h_{g}).
    \]

    Note that for all $i\geq i_{\text{final}}(h_{f})$ we have $\chi^{\text{grp}}(Z_{f}) = \chi(Z_{f}) = 0$, where the equality $\chi^{\text{grp}}(Z_{f}) = \chi(Z_{f})$ follows from the fact that $Z_{f}$ is SBR. Thus for all $i\geq i_{\text{final}}(h_{f})$ we have
    \[
        \chi\left(W_{i}(h_{f\sqcup g})\right) = \chi\left(W_{i}(h_{g})\right).
    \]

    A general compact subcover is SBR (BR) if and only if all of its components are SBR, respectively BR, thus the condition that $W_{i}(h_{f\sqcup g})$ is SBR for $i\geq i_{\text{final}}(h_{f})$ is equivalent to the condition that $W_{i}(h_{g})$ is SBR. Therefore, the terminating conditions for the process on $h_{f\sqcup g}$ are the same terminating conditions as for the process on $h_{g}$, and in particular depend only on $W_{i}(h_{g})$. As $W_{i}(h_{f\sqcup g}) = Z_{f} \sqcup W_{i}(h_{g})$ for $i\geq i_{\text{final}}(h_{f})$, the process on $h_{f\sqcup g}$ must terminate in $i_{\text{final}}(h_{g})$ steps and output
    \[
        Z_{f}\sqcup W_{i_{\text{final}}(h_{g})}(h_{g}) = Z_{f} \sqcup Z_{g}.
    \]

    The morphism $f\sqcup g$ is thus an output of the growing process on $E$ and so is an element of $R_{E}$.
    As $f$ and $g$ are injective the map $f\sqcup g$ is in $R^{\text{emb}}_{E}$. As $\chi^{\text{grp}}$ is additive we have $f\sqcup g \in R^{\text{emb}}_{E}(0)$.
    The case where $i_{\text{final}}(h_{g})\leq i_{\text{final}}(h_{f})$ is symmetric.
\end{proof}

\begin{proof}[\textbf{Part 2}]
    Let $h\in R^{\text{emb}}_{E}(0)$ and write $h:E\hookrightarrow Z$. By the third bulletin of Theorem \ref{sp res} and the characterization of $R^{\text{emb}}_{E}$, the image of $h$ meets every component of $Z$. Every component of $Z$ either has a power of $\gamma$ in its $\pi_{1}^{\text{lab}}$ (up to conjugation) or does not. Write $Z = Z_{\gamma} \sqcup Z'$ where $Z_{\gamma}$ is the sub-subcover of $Z$ comprised of all the components of $Z$ which have a non-zero power of $\gamma$ in their $\pi_{1}^{\text{lab}}$, and $Z'$ is the rest of the components, i.e. $Z'$ is the sub-subcover of $Z$ comprised of all the components of $Z$ which do not have a non-zero power of $\gamma$ in their $\pi_{1}^{\text{lab}}$.

    As the image of $h$ meets all the components of $Z$, we know that for a component $D\subseteq Z$ there is a component $C\subseteq E$ such that $h(C)\subseteq D$. By our assumption on $E$ we know $\chi^{\text{grp}}(C)=0$ so that  by Lemma \ref{EC lem} we have $\chi^{\text{grp}}(D) \leq \chi^{\text{grp}}(C) = 0$. As $\chi^{\text{grp}}(Z) =0$ and by the fact that $\chi^{\text{grp}}$ is additive we conclude that $\chi^{\text{grp}}(D) = 0$ for every component $D\subseteq Z$. In particular $\chi^{\text{grp}}(Z_{\gamma}) = \chi^{\text{grp}}(Z') = 0$.

    Lemma \ref{subgrp lem} implies that $h(E_{\gamma}) \subseteq Z_{\gamma}$ and $h(E')\subseteq Z'$.
    For the first inclusion, if $h(E_{\gamma}) \cap Z' \neq \emptyset$ then there is some component $C\subseteq E_{\gamma}$ and a component $D\subseteq Z'$ such that $h(C)\subseteq D$. As $\gamma^{k} \in \pi_{1}^{\text{lab}}(C)$ for some $k\in \mathbb{N}$ we have by Lemma \ref{subgrp lem} that $\gamma^{k} \in \pi_{1}^{\text{lab}}(D)$ which is a contradiction to $D\subseteq Z'$.

    As for the second inclusion $h(E')\subseteq Z'$, if $h(E') \cap Z_{\gamma} \neq \emptyset$ then as before there is some component $C\subseteq E'$ and a component $D\subseteq Z_{\gamma}$ such that $h(C)\subseteq D$.
    As  $D\subseteq Z_{\gamma}$ there is some $k\in \mathbb{N}$ for which $\gamma^{k}\in \pi_{1}^{\text{lab}}(D)$ (up to conjugation) while Lemma \ref{subgrp lem} together with $h(C)\subseteq D$ implies that $\pi_{1}^{\text{lab}}(D)$ contains an element that is not a conjugate of any power of $\gamma$. From this, we deduce that the conjugacy class of subgroups which is $\pi_{1}^{\text{lab}}(D)$ cannot be a conjugacy class of infinite cyclic subgroups. However, $\chi^{\text{grp}}(D) = 0$ implies that $\pi_{1}^{\text{lab}}(D)$ is a conjugacy class of infinite cyclic subgroups of $\Gamma$ which is a contradiction.

    Note that $h(E_{\gamma}) \subseteq Z_{\gamma}$ and $h(E')\subseteq Z'$ implies $h^{-1}(Z_{\gamma})\subseteq E_{\gamma}$ and \newline$h^{-1}(Z')\subseteq E'$, as $E = E_{\gamma} \sqcup E'$. This implies one can write $h = f\sqcup g$ for:
    \[
        f = h|_{E_{\gamma}} : E_{\gamma} \hookrightarrow Z_{\gamma} \;\;\;\;\; g = h|_{E'} : E' \hookrightarrow Z'.
    \]

    We are almost done: all that remains to show is that both $f$ and $g$ are the outputs of the growing process for some morphisms $E_{\gamma}\hookrightarrow Q_{\gamma}$ and $E'\hookrightarrow Q'$.
    The morphisms we use are the natural ones. As $h\in R^{\text{emb}}_{E}$, it was returned by the growing process on a morphism $r:E\hookrightarrow Q$ where $Q$ is without boundary,
    thus we let $r_{f} = r|_{E_{\gamma}} : E_{\gamma} \hookrightarrow Q$ and $r_{g} = r|_{E'} : E' \hookrightarrow Q$.

    Let us show that $W_{i_{\text{final}}(r_{f})}(r_{f}) = Z_{\gamma}$ and $W_{i_{\text{final}}(r_{g})}(r_{g}) = Z'$. To this end, let us first show that
    \[
        i_{\text{final}}(r_{f}),i_{\text{final}}(r_{g})\leq i_{\text{final}}(r).
    \]

    Consider the growing process on $r$. The process gives rise to an increasing sequence of subcovers
    \[
        Im(r) = W_{0}(r) \subseteq W_{1}(r)\subseteq ... \subseteq W_{i_{\text{final}}(r)}(r) = Z_{\gamma}\sqcup Z'.
    \]
    This shows that for all $i\leq i_{\text{final}}(r)$ we have:
    \[
        W_{i}(r) = \left(W_{i}(r) \cap Z_{\gamma}\right) \sqcup \left(W_{i}(r) \cap Z'\right).
    \]

    We know that $Z_{\gamma}\sqcup Z'$ is SBR, thus so are $Z_{\gamma}$ and $Z'$.
    In addition\newline $W_{i_{\text{final}}(r)}(r) = Z_{\gamma}\sqcup Z'$ so that
    \[
        W_{i_{\text{final}}(r)}(r) \cap Z_{\gamma} = Z_{\gamma},
    \]
    and
    \[
        W_{i_{\text{final}}(r)}(r) \cap Z'= Z'.
    \]

    Note that
    \[
        W_{0}(r)\cap Z_{\gamma}= Im(r)\cap Z_{\gamma} = Im(r_{f})  =  W_{0}(r_{f}),
    \]
    and
    \[
        W_{0}(r)\cap Z' = Im(r)\cap Z' = Im(r_{g}) = W_{0}(r_{g}).
    \]
    Thus, due to the process being deterministic, after at most $i_{\text{final}}(r)$ steps the processes on $r_{f}$ and $r_{g}$
    reach a point where $W_{i}(r_{f})$ and $W_{i}(r_{g})$ are SBR and halt.
    Thus the processes on $r_{f}$ and $r_{g}$ finish possibly sooner, meaning that
    \[
        i_{\text{final}}(r_{f}),i_{\text{final}}(r_{g})\leq i_{\text{final}}(r).
    \]

    Given that the process on $r_{f}$ and on $r_{g}$ finishes not later than the process on $r$, we conclude:
    \[
        W_{i_{\text{final}}(r_{f})}(r_{f})\subseteq Z_{\gamma} \;\;\;\;\; W_{i_{\text{final}}(r_{g})}(r_{g})\subseteq Z'.
    \]
    Note that $E_{\gamma}$ is compact, has $\chi^{\text{grp}}(E_{\gamma}) = 0$ and that $Im(r_{f})\subseteq Z_{\gamma}$ meets all of $Z_{\gamma}$'s components
    (as the image of $r$ does and $r_{f} = r|_{E_{\gamma}}$). Also note that $Z_{\gamma}$ is compact as well and that $\chi^{\text{grp}}(Z_{\gamma}) = 0$. By the key observation (Lemma \ref{key obs}) we get that $W_{i_{\text{final}}(r_{f})}(r_{f})$ is SBR.
    A similar argument with $Z_{\gamma}$ replaced by $Z'$ shows that $W_{i_{\text{final}}(r_{g})}(r_{g})$ is SBR.

    As the growing process is deterministic and as the processes on $r_{f}$ and $r_{g}$ finish not later than the process on $r$, we conclude that during the process on $r$, the subcovers
    $W_{i}(r) \cap Z_{\gamma}$ and $W_{i}(r) \cap Z'$ reach a point when they are SBR. This happens at times $i_{\text{final}}(r_{f})$ and $i_{\text{final}}(r_{g})$ respectively and at those times:
    \[
        W_{i_{\text{final}}(r_{f})}(r) \cap Z_{\gamma} = W_{i_{\text{final}}(r_{f})}(r_{f}),
    \]
    and
    \[
        W_{i_{\text{final}}(r_{g})}(r) \cap Z' = W_{i_{\text{final}}(r_{g})}(r_{g}).
    \]

    As in part 1, when $W_{i}(r) \cap Z_{\gamma}$ or $W_{i}(r) \cap Z'$ reach a point in the growing process on $r$ when they are SBR, there are no more half-blocks, half-chains, and long chains left to adjoin to them.
    Thus:
    \[
        W_{i_{\text{final}}(r)}(r) = W_{i_{\text{final}}(r_{f})}(r_{f}) \sqcup W_{i_{\text{final}}(r_{g})}(r_{g}),
    \]
    on the other hand
    \[
        W_{i_{\text{final}}(r)}(r) = Z_{\gamma} \sqcup Z'.
    \]
    As $W_{i_{\text{final}}(r_{f})}(r_{f})\subseteq Z_{\gamma}$ and $W_{i_{\text{final}}(r_{g})}(r_{g}) \subseteq Z'$ we get
    \[
        W_{i_{\text{final}}(r_{f})}(r_{f}) =  Z_{\gamma} \;\;\;\;\; W_{i_{\text{final}}(r_{g})}(r_{g}) = Z'.
    \]
    This shows that the morphisms $f:E_{\gamma}\hookrightarrow Z_{\gamma}$ and $g:E'\hookrightarrow Z'$ are results of the growing process on $E_{\gamma}$ and $E'$ respectively.
    In particular $f\in R^{\text{emb}}_{E_{\gamma}}$ and $g\in R^{\text{emb}}_{E'}$, while $\chi^{\text{grp}}(Z_{\gamma}) = \chi^{\text{grp}}(Z') = 0$ implies, by definition, that $f\in R^{\text{emb}}_{E_{\gamma}}(0)$ and $g\in R^{\text{emb}}_{E'}(0)$. In total
    \[
        h = f \sqcup g,
    \]
    for $f\in R^{\text{emb}}_{E_{\gamma}}(0)$ and $g\in R^{\text{emb}}_{E'}(0)$.
\end{proof}

\bibliographystyle{abbrv}
\bibliography{ASYMPTOTIC_INDEPENDENCE.bib}

\end{document}